\documentclass[a4paper, 11pt, reqno]{amsart}
\usepackage{amsmath,amsfonts,amssymb,amsthm,enumerate}
\usepackage{graphicx}
\usepackage{xy} \xyoption{all}

\newtheorem{thm}{Theorem}[section]
\newtheorem{cor}[thm]{Corollary}
\newtheorem{prop}[thm]{Proposition}
\newtheorem{lem}[thm]{Lemma}

\newcommand{\varep}{\varepsilon}

\theoremstyle{definition}
\newtheorem{defn}[thm]{Definition}
\newtheorem{exas}[thm]{Example}
\newtheorem{rem}[thm]{Remark}

\let\phi\varphi

\begin{document}
\title{Corners of Leavitt path algebras of finite
	graphs are Leavitt path algebras}
\maketitle
\begin{center}
G.~Abrams\footnote{Department of Mathematics, University of Colorado, Colorado Springs,
Colorado, USA. E-mail address: \texttt{abrams@math.uccs.edu}} and 
T.\,G.~Nam\footnote{Institute of Mathematics, VAST, 18 Hoang Quoc Viet, Cau Giay, Hanoi, Vietnam. E-mail address: \texttt{tgnam@math.ac.vn}

\ \ {\bf Acknowledgements}:   
The second author was supported by the Vietnam
Institute for Advanced Study in Mathematics (VIASM).} %The author  expresses their deep gratitude to Professor Gene Abrams (Department of Mathematics, University of Colorado, Colorado Springs, Coloralo, USA) for his valuable suggestions which led to the final shape of the paper.} 
\end{center}

\begin{abstract} We achieve an extremely useful description (up to isomorphism) of the Leavitt path algebra $L_K(E)$ of a finite graph $E$ with coefficients in a field $K$ as a direct sum of  matrix rings over $K$, direct sum with a corner of the Leavitt path algebra $L_K(F)$  of a graph $F$ for which  every regular vertex is the base of a loop.   Moreover, in this case one may transform the graph $E$ into the graph $F$ via some step-by-step procedure, using the ``source elimination" and  ``collapsing" processes.    We use this to establish the main result of the article,  that every nonzero corner of a  Leavitt path algebra of a finite graph is isomorphic to a Leavitt path algebra.  Indeed, we prove a more general result, to wit, that the endomorphism ring of any nonzero finitely generated projective $L_K(E)$-module is isomorphic to the Leavitt path algebra of a graph explicitly constructed from $E$.  Consequently, this yields in particular that every unital $K$-algebra which is Morita equivalent to a Leavitt path algebra is indeed isomorphic to a Leavitt path algebra.
\medskip

\textbf{Mathematics Subject Classifications}: 16S99,  05C25

\textbf{Key words}: Leavitt path algebra; Morita equivalence.
\end{abstract}

\section{Introduction and Preliminaries}
Given a (row-finite) directed graph $E$ and any field $K$, the first author and Aranda Pino in \cite{ap:tlpaoag05}, and independently Ara, Moreno, and Pardo in \cite{amp:nktfga}, introduced the \emph{Leavitt path algebra} $L_K(E)$. Leavitt path algebras generalize the Leavitt algebras $L_K(1, n)$ of \cite{leav:tmtoar}, and also contain many other interesting classes of algebras. In addition, Leavitt path algebras are intimately related to graph $C^*$-algebras (see \cite{r:ga}). During  the past fifteen years, Leavitt path algebras have become a topic of intense investigation by mathematicians from across the mathematical spectrum.  
%The study of Leavitt path algebras is related to a long tradition of associating an algebra with an appropriate combinatorial structure.
% (see \cite{aam:lpa} and the references given there),  having connections with group theory (\cite{aajz:solpaopg} and \cite{p:tipfhtg}), noncommutative algebraic geometry \cite{s:ceigmopaoq}, singularity categories (\cite{cy:hclpaagpm} and \cite{l:tilc}), Steinberg algebras (\cite{stein:agatdisa} and \cite{cfst:aggolpa}) and symbolic dynamics  (\cite{alps:fiitcolpa}, \cite{c:irolpa} and \cite{h:tdolpa}). 
For a detailed history and overview of  Leavitt path algebras we refer the reader to the survey article  \cite{a:lpatfd}.  

One of the interesting questions in the theory of Leavitt path algebras is to find relationships between graphs $E$ and $F$ such that their corresponding Leavitt path algebras are Morita equivalent. 
In  \cite{alps:fiitcolpa} the first author, Louly, Pardo, and Smith established some basic transformations of graphs which preserve isomorphism or Morita equivalence of the associated Leavitt path algebras.  
%On the other hand, as of the writing of this article, the above question is still an open ones. 
Motivated by these, along with the   ``collapsing" process introduced by S\o{}rensen in \cite{sor:gcosga}, we present here another sufficient condition for Morita equivalence between Leavitt path algebras (Theorem \ref{collapsethm}).   This equivalence result in turn provides the vehicle to establish Theorem \ref{decompthm}, which allows us to associate  (modulo some easily-handled direct summands) the Leavitt path algebra of a finite graph  with a (full) corner of the Leavitt path algebra of a graph which is a special type of extension of a ``totally looped" graph (a graph for which every non-sink vertex is the base of a loop).  

 This totally looped property turns out to play an important bridge role in the analysis, as follows.  We establish in  Corollary \ref{haircor}  that any nonzero corner (full or not) of any  Leavitt path algebra over a graph which arises as such a ``strands of hair" extension   of a totally looped graph is isomorphic to a Leavitt path algebra.  We then  use this Corollary, together with Theorem \ref{decompthm},  to establish the (perhaps surprising) generalization of the Corollary to corners of  Leavitt path algebras of all finite graphs (Theorem \ref{cornerthm}; see also Remark \ref{allidempotentsrem}).    %If $E$ is a finite graph, and $\varepsilon$ is any nonzero idempotent in $L_K(E)$, then the corner $\varepsilon L_K(E) \varepsilon$ of $L_K(E)$ is isomorphic to $L_K(F)$ for some finite graph $F$.     
As a consequence, this yields the (seemingly more-general)   Theorem \ref{End(Q)general},   which establishes that    the endomorphism  ring of any nonzero finitely generated projective module over a Leavitt path algebra is again a Leavitt path algebra.   As well, Theorem \ref{cornerthm}  easily yields that  every unital $K$-algebra that is Morita equivalent to a Leavitt path algebra is indeed isomorphic to a Leavitt path algebra.  

We now present a streamlined version of the necessary background ideas.   We refer the reader to \cite{AF} and \cite{Lam} for information about general ring-theoretic constructions, and to  \cite{AAS}  for additional information about Leavitt path algebras.

A (directed) graph $E = (E^0, E^1, s, r)$ 
%(or shortly $E = (E^0, E^1)$)
consists of two disjoint sets $E^0$ and $E^1$, called \emph{vertices} and \emph{edges}
respectively, together with two maps $s, r: E^1 \longrightarrow E^0$.  The
vertices $s(e)$ and $r(e)$ are referred to as the \emph{source} and the \emph{range}
of the edge~$e$, respectively. 
%The graph is called \emph{row-finite} if $|s^{-1}(v)|< \infty$ for all $v\in E^0$. All graphs in this paper will be assumedto be row-finite. 
A graph $E$ is \emph{finite} if both sets $E^0$ and $E^1$ are finite.  To streamline the presentation and help illuminate  the key ideas, we will focus on finite graphs throughout this article, although some of the results we establish also hold for more general graphs.  
%;  {\it we assume throughout that all graphs are finite}.  
%(or equivalently, when $E^0$ is finite, by the row-finite hypothesis).
A vertex~$v$ for which $s^{-1}(v)$ is empty is called a \emph{sink}; a vertex~$v$ for which
$r^{-1}(v)$ is empty is called a \emph{source}; a vertex~$v$ is called an \emph{isolated vertex}
if it is both a source and a sink; and a vertex~$v$ (in a finite graph)  is \emph{regular} if it is not a sink.  
%iff $0 < |s^{-1}(v)| < \infty$. 
A graph $E$ is said to be \emph{source-free} if it has no sources.   The ``trivial" graph with one vertex and no edges is denoted by $E_{triv}$.  

A \emph{path} $p = e_{1} \cdots e_{n}$ in a graph $E$ is a sequence of edges $e_{1}, \dots, e_{n}$ such that $r(e_{i}) = s(e_{i+1})$ for $i = 1, \dots, n-1$.  In this case, we say that the path~$p$ starts at the vertex $s(p) := s(e_{1})$ and ends at the vertex $r(p) := r(e_{n})$, and has \emph{length} $|p| := n$.  We consider the elements of $E^0$ to be paths of length $0$. We denote by ${\rm Path}(E)$ the set of all paths in $E$. A \emph{cycle based at} $v$ is a path $p = e_{1} \cdots e_{n}$ with $s(p) = r(p) = v$, and for which the vertices $s(e_1), s(e_2), \dots, s(e_n)$ are distinct.  A cycle $c$ is called a \emph{loop} if $|c| =1$.  A graph $E$ is \emph{acyclic} if it has no cycles. 
%An edge~$f$ is an \emph{exit} for a path $p = e_{1} \dots e_{n}$ if $s(f) = s(e_{i})$ but $f \ne e_{i}$ for some $1 \le i \le n$.

A subgraph $F$ of a finite graph $E$ is called {\it complete} in case,  for every $v\in F^0$ for which $s^{-1}_F(v) \neq \emptyset$, then $s^{-1}_F(v)  = s^{-1}_E(v)$.  Less formally:  $F$ is complete in case for every vertex $v$ of $F$, if $v$ emits at least one edge in $F$, then all edges which $v$ emits in $E$ are included in $F$.

%Let $E= (E^0, E^1)$ be a graph. 

For vertices $v, w\in E^0$, we write $v\geq w$ if there
exists a path in $E$ from $v$ to $w$, i.e.,  there exists $p\in {\rm Path}(E)$ with $s(p) = v$ and $r(p) =w$. If $v \geq w$ and $v\neq w$, then  there necessarily exists a path $q = e_1 e_2 \cdots e_t  $ from $v$ to $w$    for which the vertices $v = s(e_1), s(e_2), É,  s(e_t), w = r(e_t)$ are distinct.   (An easy induction argument shows that any path from $v$ to $w$ having minimal length will have this property.)   
%Let $S$ be s subset of $E^0$. We write $v\geq S$ if there exists a $w\in S$ such that $v\geq w$.

Let $H$ be a subset of $E^0$. $H$ is called \emph{hereditary} if for all
$v\in H$ and $w\in E^0$, $v\geq w$ implies $w\in H$.
%Denote by $\mathcal{H}_E$ the set of those subsets of $E^0$ which are both hereditary and saturated.   
For a subset $S$ of $E^0$, the {\it hereditary closure} $T(S)$ of $S$ is the (hereditary) subset $\{ w \in E^0 \ | \ s \geq w $  for some $ s\in S \}$ of $E^0$.   $H $ is called  \emph{saturated} if whenever $v$ is a regular vertex in $E^0$ with the property that $r(s^{-1}(v)) \subseteq H$, then $v\in H$.

\begin{defn}
We call the graph $E$ {\it totally looped} in case every regular vertex of $E$ is the base of at least one loop.  
\end{defn}

%As innocuous as it might seem, the following property of completely looped graphs will play a central role in our discussion.  
%\begin{cor}\label{keyproperty}
%Let $E$ be completely looped. Let $S$ be a nonempty subset of $E^0$, and let $\overline{S}$ denote the hereditary saturated closure of $S$; that is, $\overline{S}$ is the smallest subset of $E^0$ which contains $S$ and is both hereditary and saturated.  Let $w\in \overline{S}$.  Then there exists $s\in S$ for which $s\geq w$.   
%\end{cor} 
%\begin{proof}
%As described in XXX, the subset $\overline{S}$ may be constructed by an iterative procedure:   effectively, one builds the hereditary closure $S_{1,1}$ of $S$, then the saturated closure $S_{1,2}$ of $S_{1,1}$, then the hereditary closure $S_{2,2}$ of $S_{1,2}$, and so on.     Then $\overline{S}$ is precisely the union of the $S_{i,j}$.    But by Lemma \ref{everysubsetsaturated} we have that $S_{1,1}$ is saturated, so that $S_{1,1} = S_{1,2}$, which in turns yields (because $S_{1,2}$ is thereby hereditary) that $S_{i,j} = S_{1,1}$ for all $i,j$, so that $\overline{S} = S_{1,1}$, and the result follows.  
%\end{proof}

%Aranda Pino and the first author in \cite{ap:tlpaoag05}, and independently Ara, Moreno, and Pardo in \cite{amp:nktfga},  introduced Leavitt path algebras. Namely, 

For an arbitrary graph $E = (E^0,E^1,s,r)$
and any  field $K$, the \emph{Leavitt path algebra} $L_{K}(E)$ {\it of the graph}~$E$
\emph{with coefficients in}~$K$ is the $K$-algebra generated
by the sets $E^0$ and $E^1$, together with a set of variables $\{e^{*}\ |\ e\in E^1\}$,
satisfying the following relations for all $v, w\in E^0$ and $e, f\in E^1$:
\begin{itemize}
\item[(1)] $v w = \delta_{v, w} w$;
\item[(2)] $s(e) e = e = e r(e)$ and
$r(e) e^* = e^* = e^*s(e)$;
\item[(3)] $e^* f = \delta_{e, f} r(e)$;
\item[(4)] $v= \sum_{e\in s^{-1}(v)}ee^*$ for any  regular vertex $v$.
\end{itemize}
\begin{rem}\label{oneedge}  We will often use the fact that, by relation (4), if $v \in E^0$ and $s^{-1}(v)$ is a single edge (say  $s^{-1}(v)= \{f\}$), then $ff^* = v$.  
\end{rem}
If the graph $E$ is finite, then $L_K(E)$ is a unital ring having
identity $1=\sum_{v\in E^0}v$ (see, e.g., \cite[Lemma 1.6]{ap:tlpaoag05}). 
%Furthermore, it is easy to see that the mappings given by $v\longmapsto v$, for $v\in E^0$, and $e\longmapsto e^*$, $e^*\longmapsto e$ for $e\in E^1$, produce an involution on the algebra $L_K(E)$, and 
For any path $p= e_1e_2\cdots e_n$,  the element $e^*_n\cdots e^*_2e^*_1$ of $L_K(E)$ is denoted by $p^*$.   It can be shown (\cite[Lemma 1.7]{ap:tlpaoag05}) that $L_K(E)$ is  spanned as a $K$-vector space by $\{pq^* \mid p, q\in E^*, r(p) = r(q)\}$.  Indeed,  $L_K(E)$ is a $\mathbb{Z}$-graded $K$-algebra:  
$L_K(E)= \bigoplus_{n\in \mathbb{Z}}L_K(E)_n$,  where for each $n\in \mathbb{Z}$, the degree $n$ component $L_K(E)_n$ is the set \ $ \text{span}_K \{pq^*\mid p, q\in {\rm Path}(E), r(p) = r(q), |p|- |q| = n\}$. 
%Denote by $\mathcal{L}_{gr}(L_K(E))$ the lattice of graded ideals of $L_K(E)$, with order given by inclusion, and supremum and infimum given by the usual operations of ideal sum and intersection.

For any unital  ring $R$, 
%let $M_{\infty}(R)$ be the directed union of $M_n(R)$ $(n\in \mathbb{N})$, where the transition maps $M_n(R)\longrightarrow M_{n+1}(R)$ are defined by $x\longmapsto
%\left(\begin{tabular}{cc}
%x&0\\
%0&0
%\end{tabular}\right)$.
 $\mathcal{V}(R)$ denotes the set of isomorphism classes
(denoted by $[P]$) of finitely generated projective left $R$-modules.   
$\mathcal{V}(R)$ is an abelian monoid with operation
$$[P] + [Q] = [P\oplus Q]$$ for any isomorphism classes $[P]$ and $[Q]$.   On the other hand, for any directed graph
$E=(E^0, E^1, s, r)$  the monoid $M_E$ is defined as follows.   Denote by $T$ the free abelian
monoid (written additively) with generators $E^0$, and define relations on $T$ by setting
\begin{center}
	$v = \sum_{e\in s^{-1}(v)}r(e)
	%\quad\quad\quad\quad\quad\quad\quad\quad\quad (M)
	$
\end{center}
for every regular vertex $v\in E^0$.
Let $\sim_{E}$ be the congruence relation on $T$ generated by these relations.
Then $M_E $ is defined to be the quotient monoid $ T/_{\sim_E}$; we denote an element of $M_E$ by $[x]$,
where $x\in T$.  The foundational result about Leavitt path algebras for our work is the following 
%Equivalently, $V(R)$ can be viewed as the set of
%equivalence classes of idempotents in $M_{\infty}(R)$ with the operation \[[e] + [f] = \left[\left(\begin{tabular}{cc}
%e&0\\
%0&f
%\end{tabular}\right)\right] \] for idempotents $e, f\in M_{\infty}(R)$.

\begin{thm}[{\cite[Theorem 3.5]{amp:nktfga}}]\label{AMPthm}
Let $E$ be a finite graph and $K$ any field. Then the map
$[v]\longmapsto [L_{K}(E)v]$ yields an isomorphism of abelian monoids $M_E\cong \mathcal{V}(L_{K}(E))$.
%In particular, under this isomorphism, we have $[\sum_{v\in E^0}v]\longmapsto [L_{K}(E)]$.
Specifically, these two useful consequences follow immediately.  

(1)  For any regular vertex $v\in E^0$, $L_K(E)v \cong   \bigoplus_{e\in s^{-1}(v)}L_K(E)r(e)$ as left $L_K(E)$-modules.  

(2)  For any nonzero finitely generated projective left $L_K(E)$-module $Q$, there exists a sequence of (not necessarily distinct) vertices $v_1, v_2, \dots, v_\ell$ in $E$ for which $Q \cong \bigoplus_{i=1}^\ell L_K(E)v_i$; restated,  there exists a subset of (distinct) vertices $X$  of $E^0$ and positive integers $\{ n_x  \ | \  x\in X\}$  for which $Q \cong  \bigoplus_{x\in X} n_xL_K(E)x$.  
\end{thm}

We emphasize that the direct sums indicated in the above Theorem  are external direct sums.  Also, throughout, for a positive integer $n$ and left $R$-module $M$, the direct sum of $n$ copies of $M$ is denoted $nM$.

%\begin{defn}[{{The disjoint union of graphs}}]
%Let $E = (E^0, E^1, r_E, s_E)$ and $G = (G^0, G^1, r_G, s_G)$ be  graphs.	 The \textit{graph disjoint union} $E \sqcup G$ of graphs $E$ and $G$ is defined as expected: $(E \sqcup G)^0 = E^0 \sqcup G^0$, $(E \sqcup G)^1 = E^1 \sqcup G^1$, and define $r_{E \sqcup G}, s_{E \sqcup G}: (E \sqcup G)^1 \longrightarrow(E \sqcup G)^0$ by: for each $f\in (E \sqcup G)^1$,
%\begin{equation*}
%s_{E \sqcup G}(f)=  \left\{
%\begin{array}{lcl}
%s_E(f)&  & \text{if } f\in E^1  \\
%s_{G}(f)&  & \text{if } f\in G^1%
%\end{array}%
%^\right.
%\end{equation*}%
%and
%\begin{equation*} 
%\ \ \ \ \ r_{E \sqcup G}(f)=  \left\{
%\begin{array}{lcl}
%r_E(f)&  & \text{if } f\in E^1  \\
%r_G(f)&  & \text{if } f\in G^1%
%\end{array}%
%\right. \text{.}
%\end{equation*}%	
%\end{defn}

%Clearly this disjoint union process extends to any finite number of graphs.   It is easy to establish that 

We collect up in the next result some properties of Leavitt path algebras.  

\begin{prop}\label{Lpaproperties}  Let  $K$ be any field.  

(1)  $L_K(E_{triv}) \cong K$ as $K$-algebras.

(2)  Let $E_1, E_2, \dots, E_n$ be finite graphs.  The disjoint union $E_1 \sqcup E_2 \sqcup \cdots \sqcup E_n $ of graphs is defined as expected.  Then $L_K(E_1 \sqcup E_2 \sqcup \cdots \sqcup E_n ) \cong \bigoplus_{i=1}^n L_K(E_i),$ a ring direct sum as $K$-algebras.

(3)  ({\cite[Lemma 1.6.6]{AAS}}) Let $F$ be a complete subgraph of $E$.  Then $L_K(F)$ is a subalgebra of $L_K(E)$.   
\end{prop}

%\begin{lem}\label{disjunion}

%\end{lem}

%In the remainder of this section we will mention Abrams--Tomforde's interesting result \cite[Proposition 9.3]{at:iameoga} which gives that any full matrix ring over a Leavitt path algebra is also isomorphic to a Leavitt path algebra.

%As well, 
%[cf. {\cite[Definitions 1.9]{alps:fiitcolpa}}: In-split the graph]
%\begin{prop} [{\cite[Lemma 1.6.6]{AAS}}] \label{completeprop}

  %   Then $L_K(F)$ is a subalgebra of $L_K(E)$.   
%\end{prop}

We finish the introductory section by reminding the reader of  some well-known, general ring-theoretic results which will be of great importance in this analysis.   These can be found in \cite[Chapters 1, 2]{AF} and \cite[Sections 17, 18]{Lam}. Throughout,  ``ring" means ``unital ring, with $1_R \neq 0$".    If $f$ is a nonzero idempotent in the ring $R$, then the {\it corner} of $R$ by $f$ is the ring $fRf$.  (However, for notational convenience, we allow the phrase ``direct sum with a corner of ..." to include the situation where the direct summand is the zero ring; see e.g. the Abstract.) An idempotent $f \in R$ is called {\it full} in case $R = RfR$.     Rings $R$ and $S$ are {\it Morita equivalent} in case the category $R-Mod$ of left $R$-modules  and the category $S-Mod$ of left $S$-modules are equivalent.    For a left $R$-module ${}_RM$, we write $R$-endomorphisms of $M$ on the right (i.e., the side opposite the scalars); so for $f,g \in {\rm End}_R(M)$,  $(m)fg$ means ``first $f$, then $g$".  

\begin{prop}\label{End}   Let $R$ be a unital ring.

(1)  Let $e,f$  be idempotents in $R$.  Then ${\rm Hom}_R(Re,Rf) \cong eRf$ as abelian groups. 
 
(2) Suppose $P \cong  \bigoplus_{i=1}^n P_i$ as left $R$-modules.  Then 
$${\rm End}_R(P) \cong \biggl( {\rm Hom}_R(P_i, P_j)\biggr),$$
the ring of $n\times n$ matrices for which the entry in the $i$-th row, $j$-th column is an element of ${\rm Hom}_R(P_i, P_j)$, for all $1\leq i,j \leq n$.    In particular, if $P\cong  \bigoplus_{i=1}^n Re_i$, an external direct sum of the left $R$-modules $Re_i$ for idempotents $e_i$, then
$${\rm End}_R(P) \cong \biggl(e_iRe_j\biggr),$$
the ring of $n\times n$ matrices for which the entry in the $i$-th row, $j$-th column is an element of $e_iRe_j$, for all $1\leq i,j \leq n$. 
\end{prop}

\begin{thm}\label{Moritaresults}    Let $R$ and $S$ be unital rings.  

(1)  ``Morita's Theorem":  
%Proposition 18.33 in Lam
 $R$ is Morita equivalent to $S$ if and only if there exist a positive integer $n$ and a full idempotent $f\in M_n(R)$ such that $S\cong f M_n(R)f$.

(2)  Suppose $R$ decomposes as a ring direct sum $R = \bigoplus_{i=1}^n R_i$.  Then $S$ is Morita equivalent to $R$ if and only if $S$ decomposes as a ring direct sum $S = \bigoplus_{i=1}^n S_i$ where $R_i$ is Morita equivalent to $S_i$ for all $1 \leq i \leq n$.  

(3)  Let $Q$ be a nonzero finitely generated projective left $R$-module, and suppose that $Q$ is generated by $n$ elements.  Then ${\rm End}_R(Q) \cong {\rm End}_{M_n(R)}(M_n(R)q)$ for some idempotent $q \in M_n(R)$.  
\end{thm}

\section{Collapsing  at a regular vertex that is not the base of a loop}
In this section 
%we weave together ideas from three different sources: the ``Collapse" move described in \cite[Theorem 5.2]{sor:gcosga}, the ``Hedgehog" result described in \cite[Theorem 3.8]{ar:cocka}, and the ``Source Elimination" procedure described in \cite{ar:fpsmolpa}. Consequently, 
we achieve (Theorem \ref{decompthm})   an extremely useful description (up to isomorphism) of the Leavitt path algebra of any finite graph as a direct sum of  matrix rings over the coefficient field, direct sum with a corner of a Leavitt path algebra of a graph which is a special type of extension of a totally looped graph.   
%for which  every regular vertex is the base of a loop.

Here is the strategy.  
We begin  by establishing a Morita equivalence property which is similar to an analogous property of graph $C^*$-algebras. The $C^*$-result, called Move (R), was
shown in \cite[Proposition 3.2]{sor:gcosga}; it followed as a special case of the result \cite[Theorem 3.1]{cr06:csameogca}.  Move (R) applies only to very restricted configurations of vertices and edges in a graph.  By subsequently applying two additional previously-studied graph transformations (first ``in-split", then in turn  ``out-split"), both of which preserve Morita equivalence of the associated Leavitt path algebras, we are able to eliminate the restrictions on the configurations and achieve a significant generalization of Move (R), called ``collapsing"; this is the gist of Theorem \ref{collapsethm}.   We then use Theorem \ref{collapsethm} to establish Theorem \ref{decompthm}.   

%We begin this section by the following Morita equivalence property which is borrowed from \cite[Proposition 3.2]{sor:gcosga}.

\begin{defn}\label{moveRdef}
Let $E = (E^0, E^1, r, s)$ be a finite graph.   Let $w\in E^0$ be a vertex such that $w$ emits exactly one edge (call it $f$), and $f$ is not a loop, and such that $w$ receives edges from at most one vertex.  That is,  $|s^{-1}(w)| = 1$ (but $r(s^{-1}(w)) \neq \{w\})$, and either $w$ is a source vertex  or $|s(r^{-1}(w))|=1 $. 
%contains exactly one edge,  $r^{-1}(w)and  and $s(r^{-1}(w))$ are one-point sets.  
%Let $f$ denote the only edge that $w$ emits and, 
If $w$ is not a source, then denote by  $v$  the only vertex that emits to $w$. 
% Assume that $r(f) \neq w$ (i.e., that $f$ is not a loop). 
Define the ``Move (R) at $w$"  graph
$G = (G^0, G^1, r_G, s_G)$ by setting 
\[   G^0 = E^0\setminus \{w\},  \ \ \ \ G^1 = (E^1\setminus (r^{-1}(w) \cup \{f\})) \cup \{[ef] \mid e\in r^{-1}(w)\},\] where range and source maps extend those of $E$, and satisfy $r_G([ef]) = r(f)$ and $s_G([ef]) = s(e) = v$.   (Note that, in case $w$ is a source, then  $G^1$ is simply $E^1 \setminus \{f\}$.)
\end{defn}

\begin{rem}
We note that if there is a loop based at $w$, then the Move (R) construction at $w$ does not yield a well-defined graph, because in that case the range function $r_G$ would be undefined at  edges of the form $[ef]$ (since that vertex $w$ has been  eliminated).  
\end{rem}

\begin{prop}\label{moveRprop}
Let $K$ be any field.  Let $E$ be a finite graph, let $w\in E^0$ be a vertex of the type described in Definition \ref{moveRdef}, and let $G$ be the corresponding Move (R) at $w$  graph.    Then $L_K(E)$ is Morita equivalent to $L_K(G).$
\end{prop}

\begin{proof} We construct a $K$-algebra homomorphism $$\psi: L_K(G)\longrightarrow
L_K(E)$$ given on the generators of the free $K$-algebra $K\langle u, g, g^*\mid u\in G^0, g\in G^1\rangle$ as follows: 
\begin{equation*}
\psi(u) = u,    \ \ \  \ \ 
\psi(g)=  \left\{
\begin{array}{lcl}
ef&  & \text{if } g = [ef]\ (e\in r^{-1}(w)) , \\
g&  & \text{otherwise \ \ }%
\end{array}%
\right.
\end{equation*}%
and \begin{equation*}
\psi(g^*)=  \left\{
\begin{array}{lcl}
f^*e^*&  & \text{if } g = [ef]\ (e\in r^{-1}(w)) , \\
g^*&  & \text{otherwise. \ \ }%
\end{array}%
\right.
\end{equation*}%
To ensure that the map $\psi$ induces  a $K$-algebra  homomorphism from $L_K(G)$ to $L_K(E)$, we must verify that all elements of the following forms:
\medskip

$uu' - \delta_{u, u'}u$ for all $u, u'\in G^0$,
\smallskip

$s_{G}(g)g - g$ and $g-gr_{G}(g)$ for all $g\in G^1$,
\smallskip

$r_{G}(g)g^*-g^*$ and $g^*-g^*s_{G}(g)$ for all $g\in G^1$,
\smallskip

$g^*h - \delta_{g, h}r_{G}(g)$ for all $g, h\in G^1$,
\smallskip

$u - \sum_{g\in s_{G}^{-1}(u)}gg^*$ for a regular vertex $u\in G^0$\\

\noindent
are in the kernel of $\psi$. Here we verify only the last one, since the first four  can be established easily.    Let $u\in G^0$ be a regular vertex. If $w$ is a source then $s_{G}^{-1}(u) = s^{-1}(u)$, and so $\psi(u - \sum_{g\in s_{G}^{-1}(u)}gg^*)= u - \sum_{g\in s^{-1}(u)}gg^* = 0$, as desired. Suppose $w$ is not a source. Then by hypothesis $w$ receives from only one vertex, $v$ say. If $u \neq v$ then the statement is proved by a similar argument to that given above. So consider the case when $u = v$. We note that $ff^* = w$ by Remark \ref{oneedge}, and
$s_G^{-1}(v) = (s^{-1}(v)\setminus r^{-1}(w)) \sqcup \{[ef]\mid e\in r^{-1}(w)\}.$  Hence
\begin{equation*}
\begin{array}{rcl}
\psi(v - \sum_{g\in s_{G}^{-1}(v)}gg^*) &=& v - \sum_{g\in s^{-1}(v)\setminus r^{-1}(w)}gg^* - \sum_{e\in r^{-1}(w)}(ef)(f^*e^*)\\
&=& v - \sum_{g\in s^{-1}(v)\setminus r^{-1}(w)}gg^* - \sum_{e\in r^{-1}(w)}ee^*\\
&=& v - \sum_{g\in s^{-1}(v)}gg^* = 0.
\end{array}
\end{equation*}

We next prove that $\psi: L_K(G) \to L_K(E)$ is injective. To the contrary, suppose 
%$\psi$ is not injective, i.e., 
there exists a nonzero element $x\in \ker(\psi)$. Then, by the Reduction Theorem (see, e.g., \cite[Theorem 2.2.11]{AAS}), there exist $a, b\in L_K(G)$ such that 
%one of these two possibilities (or both) occur: 
either $axb = u \neq 0$ for some $u\in G^0$, or $axb = p(c)\neq 0,$ where $c$ is a cycle in $G$ and $p(x)$  is a nonzero polynomial in
$K[x, x^{-1}]$. 

In the first case, since $axb\in \ker(\psi)$, this would imply that $u =\psi(u) = 0$ in $L_K(E)$; but each vertex is well-known to be a nonzero element inside the Leavitt path algebra, a contradiction.     

So we are in the second case:   there exists a cycle $c$ in $G$ such that $axb = \sum^m_{i= -n}k_ic^i$, where $k_i\in K$ and we interpret $c^i$ as $(c^*)^{-i}$ for negative $i$, and we interpret $c^0$ as $u := s(c)$.
We then have $\sum^m_{i= -n}k_i\psi(c)^i= \psi(axb) = 0$ in $L_K(E)$.
We write $c= g_1 \cdots g_m$, where $g_i\in G^1$. If  none of the $g_i$'s are in $\{[ef] \mid e\in r^{-1}(w)\}$, then $\psi(c) = \psi(g_1)\cdots \psi(g_m) = g_1 \cdots g_m$, so $\psi(c)$ is also a cycle in $E$. Otherwise, there
exists $1\leq k\leq m$ such that $g_k = [ef]$ for some $e\in r^{-1}(w)$. We then have that $g_i\in E^1\setminus (r^{-1}(w) \cup \{f\})$ for all $i \neq k$, since $c$ is a cycle. It shows that $\psi(c) = \psi(g_1...g_{k-1}[ef]g_{k+1}...g_{m}) = g_1...g_{k-1}efg_{k+1}...g_{m}$, so $\psi(c)$ is a cycle in $E$. Thus, in any case $\psi(c)$ is a cycle in $E$.  But the $\mathbb{Z}$-grading in $L_K(E)$ shows that an equation of the type $\sum^m_{i= -n}k_i\psi(c)^i= \psi(axb) = 0$ cannot hold in $L_K(E)$. This shows that $\psi$ is injective, so $L_K(G)$ is isomorphic to $Im(\psi)$.

Let $\epsilon =\psi(1_{L_K(G)}) = \sum_{u\in E^0, u\neq w}u$. We claim that $Im(\psi) = \epsilon L_K(E)\epsilon$.  Clearly $Im(\psi) \subseteq \epsilon L_K(E)\epsilon$. To show the other inclusion, we only need to verify that for all nonzero monomials   
$x\in \epsilon L_K(E)\epsilon$,
$\psi^{-1}(x) \neq \emptyset$. We can express such a monomial $x$ of the form $x = pq^*$, where $p = g_1\cdots g_m$ and $q= f_1 \cdots f_n$ are  paths in $E$ such that $r(p) = r(q)$ and $s(p), s(q)\neq w$. Consider the following three cases. (We note in advance that the last two possibilities of Case 1 and Case 3 cannot occur when $w$ is a source.)

\emph{Case} 1. $r(p) = r(q) = r(f)$. If both $g_m$ and $f_n$ are different from $f$, then $g_i$ and $h_j\in E^1\setminus (r^{-1}(w) \cup \{f\})$, so $x = pq^*\in L_K(G)$ and $\psi(x) = x.$ If $g_m = f$ then $f_n\neq f$ (otherwise, since $ff^* = w$, we get that $x = g_1...g_{m-1}ff^*f^*_{n-1}...f^*_1 = g_1...g_{m-1}f^*_{n-1}...f^*_1$ and $r(g_1...g_{m-1}) = r(f_1...f_{n-1})\neq r(f)$, a contradiction). We then immediately obtain that $m\geq 2$, $g_{m-1}\in r^{-1}(w)$ (since $s(p) = s(g_1)\neq w$), and $g_i, f_j\in E^1\setminus (r^{-1}(w) \cup \{f\})$ for all $1\leq i\leq m-2$ and $1\leq j\leq n$. Set $\alpha = g_1...g_{m-2}[g_{m-1}f]f^*_{n}...f^*_1\in L_K(G)$. We have that $\psi(\alpha) = g_1...g_{m-2}g_{m-1}ff^*_{n}...f^*_1 = x.$ 
%Similar to the case when $f_n = f$. 
 If $f_n = f$ then $g_m$ is different from $f$ (otherwise, since $ff^* =w$, $x = g_1...g_{m-1}ff^*f^*_{n-1}...f^*_1 = g_1...g_{m-1}f^*_{n-1}...f^*_1$ and $r (g_1...g_{m-1}) = r(f_1...f_{n-1}) \neq r(f)$, a contradiction). We then have that, for $n \geq 2, f_{n-1}$ is in $r^{-1}(w)$ (since $s(q) = s(f_1) \neq w$), and $g_i, f_j$ are not in $r^{-1}(w) \cup \{f\}$ for all $1 \leq i \leq m$ and $1 \leq j \leq n-2$. Define $ \beta := g_1 ... g_m [f_{n-1}f]^*f^*_{n-2}...f^*_1 \in L_K(G)$. Then we have $\phi(\beta) = g_1 ... g_m f^*f^*_{n-1}f^*_{n-2}...f^*_1= x$. 

\emph{Case} 2. $r(p) = r(q) \notin \{r(f), w\}$. We immediately get that both $g_i$ and $h_j\in E^1\setminus (r^{-1}(w) \cup \{f\})$ for all $i, j$, so $x = pq^*\in L_K(G)$ and $\psi(x) = x.$

\emph{Case} 3. $r(p) = r(q) = w$. We have that $g_m$ and $f_n$ are in $r^{-1}(w)$ and $g_i\ (1\leq i\leq m-1), f_j\ (1\leq j\leq n-1) \in E^1\setminus (r^{-1}(w) \cup \{f\})$, and so
%\begin{equation*}
%\begin{array}{rcl}
$$x = g_1...g_{m}f^*_{n}...f^*_1
 = g_1...g_{m}.w.f^*_{n}...f^*_1 = 
g_1...g_{m}ff^*f^*_{n}...f^*_1=\psi(\beta)$$
%\end{array}
%\end{equation*}
for $\beta := g_1...g_{m-1}[g_nf][f_nf]^*f^*_{n-1}...f^*_1\in L_K(G)$. 

In any case we always have $\psi^{-1}(x) \neq \emptyset$, so $Im(\psi) = \epsilon L_K(E)\epsilon$, thus showing that $L_K(G)$ is isomorphic to $\epsilon L_K(E)\epsilon$.

To establish the Morita equivalence, we  show that $L_K(E) = L_K(E)\epsilon L_K(E)$ (see the $n=1$ case of Theorem \ref{Moritaresults}(1)). It is enough to show that $w$ is in $ L_K(E)\epsilon L_K(E)$. Since $r(f)$ is in the ideal $L_K(E)\epsilon L_K(E)$, the edge $f$ is in $ L_K(E)\epsilon L_K(E)$. Then $w = ff^*\in  L_K(E)\epsilon L_K(E)$. This proves that $L_K(E)\epsilon L_K(E) = L_K(E)$. Hence $L_K(G)$ is Morita equivalent to $L_K(E)$, finishing the proof.	
\end{proof}

%Our second and third Morita equivalence properties restrict somewhat cumbersome machinery that were built in \cite[Sections 1 and 3]{alps:fiitcolpa}. The following definition is restricted from \cite[Definitions 1.9]{alps:fiitcolpa}.

The key Morita equivalence result for us (Theorem \ref{collapsethm}) is inspired by \cite[Theorem 5.2]{sor:gcosga}. To achieve it, we  show that  Move (R) may be applied at vertices that are more general than those given in Definition \ref{moveRdef}, and that the corresponding Leavitt path algebras are Morita equivalent.  
% at a  regular vertex which is not a base of a loop preserves the Morita equivalence property of the associated Leavitt path algebras.   
Following \cite{sor:gcosga}, we call this generalization of Move (R) a  ``collapse".

\begin{defn}\label{collapsedef}  [Collapse at a regular vertex which is not the base of a loop]
Let $E = (E^0, E^1, r, s)$ be a finite graph, and let $v\in E^0$ be a regular vertex which is not the base of a loop. Define the ``collapse at $v$"  graph
$G = (G^0, G^1, r_G, s_G)$ by setting 
\[  G^0 = E^0\setminus \{v\},  \ \ \ 
   G^1 = (E^1\setminus (r^{-1}(v) \cup s^{-1}(v))) \cup \{[ef] \mid e\in r^{-1}(v),\ f\in s^{-1}(v)\},\] where range and source maps extend those of $E$, and satisfy $r_G([ef]) = r(f)$ and $s_G([ef]) = s(e)$.   (We note that, in case $v$ is a source, then  $G^1$ is simply $E^1 \setminus s^{-1}(v)$.)
\end{defn}

\begin{rem}  As with Move (R), the requirement that there be no loop based at the collapsing vertex is necessary so that the collapsing process gives a well-defined graph.
\end{rem}

Specific examples of this collapsing process are given below in Example \ref{collapsenotunique}.

\medskip

We extend Proposition \ref{moveRprop} in two stages.  In the first stage, we show that the requirement that the collapsing vertex receive edges from at most one vertex can be eliminated (Proposition \ref{collapseoneedgeprop}).

\begin{defn}[{\cite[Definitions 1.9]{alps:fiitcolpa}}: the ``in-split" graph]\label{insplitdef}
Let $E = (E^0, E^1, r, s)$ be a graph and $v\in E^0$ a vertex that is not a source. Partition $r^{-1}(v)$ into a finite number, say $n$, of disjoint nonempty subsets $\mathcal{E}_1, \mathcal{E}_2, ..., \mathcal{E}_n$. We form the \textit{in-split graph} $E_{is} = (E^0_{is}, E^1_{is}, r_{is}, s_{is})$ from $E$ using the partition $\{\mathcal{E}_i\mid i = 1, ..., n\}$ as follows: $E^0_{is} = (E^0\setminus \{v\}) \cup \{v_1, v_2, ..., v_n\}$,

\begin{center}
$E^1_{is} = \{e_1, e_2, ..., e_n\mid e\in E^1, s(e) = v\} \cup \{f\mid f\in E^1\setminus s^{-1}(v)\}$,
\end{center} 
and define $r_{is}$, $s_{is}: E_{is}^1 \longrightarrow E_{is}^0$ by setting 
$s_{is}(e_j) = v_j$,  $s_{is}(f) = s(f)$, 	and

\begin{equation*}
r_{is}(x)=  \left\{
\begin{array}{lcl}
r(f)&  & \text{if } x = f\notin r^{-1}(v)\\
v_i&  & \text{if } x =f\in r^{-1}(v) \text{ and } f\in \mathcal{E}_i\\
r(e)&  & \text{if } x = e_j \text{ and } e\notin r^{-1}(v)\\
v_i&  & \text{if } x = e_j,\ e\in r^{-1}(v) \text{ and } e\in \mathcal{E}_i
\end{array}%
.
\right.
\end{equation*}% 

\end{defn}

\begin{prop}[essentially {\cite[Proposition 1.11 and Corollary 3.9]{alps:fiitcolpa}}] \label{insplitprop}
Let $K$ be any field. Let $E$ be a finite graph and $v\in E^0$ a vertex that is not a source. Then $L_K(E)$ is Morita equivalent to $L_K(E_{is})$.
\end{prop}

\begin{proof}   The quoted result \cite[Corollary 3.9]{alps:fiitcolpa} applies to constructions more general than the in-split construction; however, the tools used to prove \cite[Corollary 3.9]{alps:fiitcolpa} only allow for the desired conclusion when the field $K$ is infinite.   
Accordingly, we provide here a short proof of Proposition \ref{insplitprop} which holds for all fields.  

	We define the  elements $\{Q_u \ | \ u\in E^0\}$ and $\{T_e, T_{e^*} \ | \ e\in E^1\}$ of $L_K(E_{is})$  by setting 
	\begin{equation*}
	Q_u=  \left\{
	\begin{array}{lcl}
	v_1&  & \text{if } u = v , \\
	u&  & \text{otherwise \ \ ,}%
	\end{array}%
	\right.
	\end{equation*}%
	\medskip
	%For each $e\in E^1$, define	
	\begin{equation*}
	T_e=  \left\{
	\begin{array}{lcl}
	\sum_{f\in s^{-1}(v)}ef_if^*_1&  & \text{if } e\in (r^{-1}(v)\cap \mathcal{E}_i)\setminus s^{-1}(v),\ s^{-1}(v) \neq \emptyset\\
	e&  & \text{if } e\in (r^{-1}(v)\cap \mathcal{E}_i)\setminus s^{-1}(v),\ s^{-1}(v) = \emptyset\\
	\sum_{f\in s^{-1}(v)}e_1f_if^*_1&  & \text{if } e\in r^{-1}(v)\cap \mathcal{E}_i\cap s^{-1}(v)\\
	e&  & \text{if } e\notin r^{-1}(v)  \ \ , \ \ \ \ \   \text{and}
	\end{array}%
	\right.
	\end{equation*}%
	\medskip
	%and 
	\begin{equation*}
	T_{e^*}=  \left\{
	\begin{array}{lcl}
	\sum_{f\in s^{-1}(v)}f_1f^*_ie^*&  & \text{if } e\in (r^{-1}(v)\cap \mathcal{E}_i)\setminus s^{-1}(v),\ s^{-1}(v) \neq \emptyset\\
	e^*&  & \text{if } e\in (r^{-1}(v)\cap \mathcal{E}_i)\setminus s^{-1}(v),\ s^{-1}(v) = \emptyset\\
	\sum_{f\in s^{-1}(v)}f_1f^*_ie^*_1&  & \text{if } e\in r^{-1}(v)\cap \mathcal{E}_i\cap s^{-1}(v)\\
	e^*&  & \text{if } e\notin r^{-1}(v)  \ \ .
	\end{array}%
	%\text{.}
	\right.
	\end{equation*}%
	\medskip
	
	\noindent
	By repeating verbatim the corresponding argument in
	the proof of \cite[Proposition 1.11]{alps:fiitcolpa}, there exists an $K$-algebra homomorphism $\pi: L_K(E)\longrightarrow L_K(E_{is})$, which maps $u\longmapsto Q_u$, $e\longmapsto T_e$ and $e^*\longmapsto T_{e^*}$, such that $\pi(L_K(E)) = \pi(1_{L_K(E)})L_K(E_{is})\pi(1_{L_K(E)})$. Note that 
	$\pi(1_{L_K(E)}) = v_1 +\sum_{u\in E^0\setminus \{v\}}u =: \epsilon$.
	
	Since $Q_u$ has degree $0$, $T_e$ has degree $1$, and $T_{e^*}$ has degree $-1$ for all $u\in E^0$ and $e\in E^1$, $\pi$ is thus a $\mathbb{Z}$-graded homomorphism, whence the injectivity of $\pi$ is guaranteed by \cite[Theorem 4.8]{tomf:utaisflpa}. So we have $L_K(E) \cong \epsilon L_K(E_{is}) \epsilon$.   
	
	To obtain the Morita equivalence, we invoke Theorem \ref{Moritaresults}(1) (with $n=1$); so we  need only establish that $\epsilon$ is full, i.e., that  $L_K(E_{is}) = L_K(E_{is})\epsilon L_K(E_{is})$.  It is enough to show that $v_i$ is in $L_K(E_{is})\epsilon L_K(E_{is})$ for all $2\leq i\leq n$. We consider the following cases.
	
	\textit{Case} 1. $\mathcal{E}_i$ does not contain loops for all $i$. Then, for each $1\leq i\leq n$, there exists $f_i\in r^{-1}(v)\cap \mathcal{E}_i$ such that $s(f_i)\neq v$. Therefore, $s_{is}(f_i) = s(f_i)\in E^0\setminus \{v\}$ and $r_{is}(f_i) = v_i$ for all $i$. This implies that $s_{is}(f_i)$ is in the ideal $L_K(E_{is})\epsilon L_K(E_{is})$, and so the edge $f_i = s_{is}(f_i)f_i$ 
	is in $L_K(E_{is})\epsilon L_K(E_{is})$. Then $v_i = r_{is}(f_i) =f^*_if_i\in  L_K(E_{is})\epsilon L_K(E_{is})$ for all $i$. This proves that $L_K(E_{is})\epsilon L_K(E_{is}) = L_K(E_{is})$ in this case. 
	
	\textit{Case} 2. There exists $1\leq i\leq n$ such that $\mathcal{E}_i$ contains a loop. We may, without loss of generality, that there exists $1\leq k\leq n$ such that $\mathcal{E}_i$ $(1\leq i\leq k)$ contains a loop and $\mathcal{E}_j$ $(k+1\leq j\leq n)$ does not contain loops. Then, similar to Case 1, we get that $v_j$ is in the ideal $L_K(E_{is})\epsilon L_K(E_{is})$ for all $k+1\leq j\leq n$. For each $1\leq i\leq k$, there exists $f_i\in\mathcal{E}_i$ such that $s(e^i) = r(e^i)= v$. We have that $s_{is}(e^i_1) = v_1$ and $r_{is}(e^i_1) = v_i$ for all $1\leq i\leq k$. Since $v_1$ is in the ideal $L_K(E_{is})\epsilon L_K(E_{is})$, the edge $e^i_1 = v_1e^i_1$ 
	is in $L_K(E_{is})\epsilon L_K(E_{is})$. Then $v_i = r_{is}(e^i_1) = (e^i_1)^*e^i_1$ is in the ideal $L_K(E)\epsilon L_K(E)$ for all 
	$1\leq i\leq k$. This implies that $L_K(E_{is})\epsilon L_K(E_{is}) = L_K(E_{is})$ in this case as well. 
	
	Thus $L_K(E)$ is Morita equivalent to $L_K(E_{is})$, finishing the proof.
\end{proof}

%\medskip

Here is the first generalization of Proposition \ref{moveRprop}, in which we remove the hypotheses that the collapsing vertex receives edges from at most one other vertex.  

%Using Propositions \ref{moveRprop} and \ref{insplitprop}, we obtain the following useful Morita equivalence, which is the algebraic analog of \cite[Lemma 5.1]{sor:gcosga}.

\begin{prop}\label{collapseoneedgeprop}
Let $E = (E^0, E^1, r, s)$ be a finite graph, and let $v\in E^0$ be a regular vertex which emits precisely one edge, $f_0$ say. Assume that $r(f_0)\neq v$ (i.e., that $f_0$ is not a loop).  Let $G$ be the ``collapse at $v$"  graph.   
% Define a graph
%$G = (G^0, G^1, r_G, s_G)$ by setting 
%\[      G^0 = E^0\setminus \{v\},   \ \ \    G^1 = (E^1\setminus (r^{-1}(v) \cup s^{-1}(v))) \cup \{[ef_0] \mid e\in r^{-1}(v)\},\] where range and source maps extend those of $E$, and satisfy $r_G([ef_0]) = r(f_0)$ and $s_G([ef_0]) = s(e)$. 
Let $K$ be any field. Then $L_K(E)$ is Morita equivalent to $L_K(G)$.
\end{prop}
\begin{proof}
If $v$ receives from at most one vertex (i.e, $|s(r^{-1}(v))| \leq 1$) then the statement follows immediately from Proposition \ref{moveRprop}. We assume then that $v$ receives from the vertices $\{u_1, u_2, ..., u_n\}$, where $n\geq 2$. Partition $r^{-1}(v)$ into disjoint nonempty subsets $\mathcal{E}_i = r^{-1}(v)\cap s^{-1}(u_i)$ for all $i = 1, 2, ..., n$. By using Proposition \ref{insplitprop}  at $v$ according to the partition $\{\mathcal{E}_i\}$, we get that $L_K(E)$ is Morita equivalent to $L_K(E_{is})$.  But  the graph $E_{is}$ is the graph $E$, with $v$ replaced by $n$ vertices $v_1, v_2, ..., v_n$,  each receiving from exactly one of the vertices $v$ received from,  and each emitting one edge to $r(f_0)$.  So we may apply Move (R) 
% Proposition \ref{moveRprop} 
step by step at $v_1$  through $v_n$.  At each step, the Morita equivalence of the corresponding Leavitt path algebra is ensured by Proposition \ref{moveRprop}.   But this sequence of graph transformations, which first in-splits at $v$ and then performs Move (R) at each of the $v_i$, yields precisely the collapse at $v$ graph $G$.   
%So  this yields that $L_K(E_{is})$ is Morita equivalent to $L_K(G)$. 
Hence  $L_K(E)$ is Morita equivalent to $L_K(G)$.
\end{proof}

%Our fourth Morita equivalence property was described in \cite{aalp:tcqflpa}.   
%restricts a cumbersome machinery that was built in \cite[Section 1]{alps:fiitcolpa}. The following definition is restricted from \cite{aalp:tcqflpa} (also, see \cite{alps:fiitcolpa}). The following definition follows as a specific case of  \cite[Definition 2.6]{aalp:tcqflpa}.

We now show how to eliminate the restriction in Proposition \ref{collapseoneedgeprop} which requires that the collapsing vertex emits just one edge.

\begin{defn}[{\cite[Definition 2.6]{aalp:tcqflpa}}: the ``out-split" graph]\label{outsplitdef}
Let $E = (E^0, E^1, r, s)$ be a graph and $v\in E^0$ a vertex that is not a sink. Partition $s^{-1}(v)$ into a finite number, say $n$, of disjoint nonempty subsets $\mathcal{E}_1, \mathcal{E}_2, ..., \mathcal{E}_n$.  We form the \textit{out-split graph} $E_{os} = (E^0_{os}, E^1_{os}, r_{os}, s_{os})$ from $E$ using the partition $\{\mathcal{E}_i\mid i = 1, ..., n\}$ as follows:
$E^0_{os} = (E^0\setminus \{v\}) \cup \{v^1, v^2, ..., v^n\}$,
	
\begin{center}
$E^1_{os} = \{e^1, e^2, ..., e^n\mid e\in E^1, r(e) = v\} \cup \{f\mid f\in E^1\setminus r^{-1}(v)\}$,
\end{center} 
and define $r_{os}$, $s_{os}: E_{os}^1 \longrightarrow E_{os}^0$ by setting	$r_{os}(e^j) = v^j$,  $r_{os}(f) = r(f)$, 	and
	
\begin{equation*}
s_{os}(x)=  \left\{
\begin{array}{lcl}
s(f)&  & \text{if } x = f\notin s^{-1}(v)\\
v^i&  & \text{if } x =f\in s^{-1}(v) \text{ and } f\in \mathcal{E}_i\\
s(e)&  & \text{if } x = e^j \text{ and } e\notin s^{-1}(v)\\
v^i&  & \text{if } x = e^j,\ e\in s^{-1}(v) \text{ and } e\in \mathcal{E}_i
\end{array}%
.
\right.
\end{equation*}% 
\end{defn}

%As an immediate corollary of \cite[Theorem 2.8]{aalp:tcqflpa}, we get the following note.

\begin{prop}[{\cite[Theorem 2.8]{aalp:tcqflpa}}] \label{outsplitprop}
Let $K$ be any field. Let $E$ be a finite graph and $v\in E^0$ a regular vertex. Then $L_K(E)\cong L_K(E_{os})$ as $\mathbb{Z}$-graded $K$-algebras. In particular, $L_K(E)$ is Morita equivalent to $L_K(E_{os})$.
\end{prop}

We are now in position to achieve the key Morita equivalence result.  

\begin{thm} Let $K$ be any field. \label{collapsethm}
Let $E$ be a finite graph, let $v \in 
E^0$ be a regular  vertex which is not the base of a loop, and let $G$ be the ``collapse at $v$" graph.    Then $L_K(E)$ is Morita equivalent to $L_K(G)$.
\end{thm}
\begin{proof}
If $|s^{-1}(v)| = 1$ then the result  follows immediately from Proposition \ref{collapseoneedgeprop}. We assume that $e_1, e_2, ..., e_n$ are the edges with source $v$, where $n\geq 2$. Partition $s^{-1}(v)$ into disjoint nonempty subsets $\mathcal{E}_i = \{e_i\}$ for $i = 1, 2, ..., n$. Applying Proposition \ref{outsplitprop}   at $v$ according to the partition $\{\mathcal{E}_i\}$ of $s^{-1}(v)$, we get that $L_K(E)$ is isomorphic to $L_K(E_{os})$. 
%, where the graph $E_{os}$ is defined by: $E^0_{os} = E^0\setminus \{v\}\cup\{v^1, v^2, ..., v^n\}$,
%
%\begin{center}
%$E^1_{os} = E^1\setminus (r^{-1}(v)\cup s^{-1}(v))\cup \{f^i\mid f\in r^{-1}(v), i = 1, ..., n\} \cup \{\overline{e_i}\mid i = 1, .., n\}$
%\end{center}
%with range and source map that agree with those of $E$ when the edge is in $E^1$ but with
%
%\begin{center}
%$r_{os}(f^i) = v^i$, $r_{os}(\overline{e_i})= r(e_i)$, $s_{os}(f^i) = s(f)$ and $s_{os}(\overline{e_i})= v^i$.
%\end{center}
Since $v$ is not the base of a loop,  for each $1 \leq i \leq n$, $v^i$ emits exactly one edge $e_i$, and $e_i$  is not a loop.  In particular,  each $v^i$ satisfies the hypotheses of Proposition \ref{collapseoneedgeprop}.  So we may apply the collapsing process  
% Proposition \ref{moveRprop} 
step by step at $v^1$  through $v^n$.  At each step, the Morita equivalence of the corresponding Leavitt path algebra is preserved by Proposition \ref{collapseoneedgeprop}.   But this sequence of graph transformations, which first out-splits at $v$ and then collapses at each of the $v^i$, yields precisely the collapse at $v$ graph $G$.   
%So  this yields that $L_K(E_{is})$ is Morita equivalent to $L_K(G)$. 
Hence  $L_K(E)$ is Morita equivalent to $L_K(G)$, as desired.
%and conclude that    $L_K(E_{os})$ is Morita equivalent to $L_K(G)$.    Therefore, $L_K(E)$ is Morita equivalent to $L_K(G)$.
\end{proof}

\begin{defn}[{\cite[Definition 1.2]{alps:fiitcolpa} and \cite[Notation 2.4]{ar:fpsmolpa}}] \label{sourceelimdef} 
Let $E = (E^0, E^1, r, s)$ be a graph, and let $v\in E^0$ be a source. We form the \emph{source
elimination} graph $E_{\setminus v}$ of $E$ as follows:

\begin{center}
$(E_{\setminus v})^0 = E^0\setminus \{v\}$, $(E_{\setminus v})^1 =
E^1\setminus s^{-1}(v)$, $s_{E_{\setminus v}} = s|_{(E_{\setminus
v})^1}$ and $r_{E_{\setminus v}} = r|_{(E_{\setminus v})^1}$. 
\end{center}
In other words, $E_{\setminus v}$ denotes the graph constructed from $E$ by deleting $v$ and all of edges in $E$ emitting from $v$.
\end{defn}

We note that the source elimination process is allowed at isolated vertices.   If $v$ is a source vertex in a graph $E$, and $v$ is not an isolated vertex, then clearly the source elimination process at $v$ coincides with the ``collapsing at $v$" move. So Theorem \ref{collapsethm} immediately gives  the following previously-established result. 

\begin{cor}[{\cite[Lemma 4.3]{ar:fpsmolpa}}]\label{AraRangasourceelimlem}
Let $E$ be a finite graph and $K$ any field. If $v\in E^0$ is a source vertex which is not isolated, then $L_K(E)$ is Morita equivalent to $L_K(E_{\setminus v})$. 
\end{cor}

We note  that Theorem \ref{collapsethm}   yields \cite[Lemma 4.4]{ar:fpsmolpa} as well.  

\medskip

Let $E$ be a finite graph. If $E$ is acyclic, then repeated application of the source
elimination process to $E$ yields the empty graph. On the other hand, if $E$ contains a cycle,
then repeated application of the source elimination process will yield a source-free graph
$E_{sf}$ which necessarily contains a cycle.

Consider the sequence of graphs which arises in some step-by-step process of source
eliminations
\[E= E_0\rightarrow E_1\rightarrow\cdots\rightarrow E_i\rightarrow\cdots
\rightarrow E_t = E_{sf}.\] To avoid defining a graph to be the empty set, we define
$E_{sf}$ to be the graph $E_{triv}$ (consisting of one vertex and no edges) in case
$E_{t-1} = E_{triv}$.

\begin{rem}\label{Esfunique}
Although there in general are many different orders in which a step-by-step source elimination process can be carried out, the resulting source-free subgraph $E_{sf}$ is always the same
(see, e.g., \cite[Lemma 3.13]{anp:lpahugn}). 
\end{rem}

\begin{rem}\label{isolatedverticesrem}   For a finite graph $E$, we perform a sequence of source eliminations  $E= E_0\rightarrow E_1\rightarrow\cdots\rightarrow E_i\rightarrow\cdots
\rightarrow E_t = E_{sf}$.     By Remark \ref{Esfunique} we may assume that the final $k$ steps in the process involve eliminating any isolated vertices which may arise.   
%In such a process, we keep track of the number $k$ of isolated vertices which are source-eliminated.  
The non-negative integer $k$ is precisely the number of vertices of $E$ which are sinks in $E$, but which are not in $E_{sf}^0$.   As observed in \cite{anp:lpahugn}, $k$ may also be viewed as the number of sinks $u$ in $E$ for which every path $p \in {\rm Path}(E)$ having $r(p)=u$ contains no closed subpath.    Let $V_k$ denote the graph with $k$ vertices and no edges.   Then repeated application of  Lemma \ref{AraRangasourceelimlem} gives that $L_K(E)$ is Morita equivalent to $L_K(V_k \sqcup E_{sf})$ (when $E_{sf}$ is nontrivial), and to $L_K(V_k)$ (when $E_{sf} $ is the trivial graph $ V_1$).  
\end{rem}   

\begin{rem}\label{collapserem}
For a finite graph $E$, we form the (uniquely-determined) graph $E_{sf}$.    In case $E_{sf}$ is not the trivial graph, we may then perform a step-by-step process in which we produce a sequence of graphs $$E_{sf} = F_0\rightarrow F_1\rightarrow\cdots\rightarrow\cdots
\rightarrow F_\ell := F,$$
where each $F_{i+1}$ is formed from $F_i$ by performing a collapsing at some vertex $v$ of $F_i$ which is not the base of a loop.   (We note that by the construction of $E_{sf}$  there will be no isolated vertices in any of the $F_i$.)   In this way,  $E_{sf}$ is transformed to a totally looped graph $F$ in which there are no isolated vertices.   
\end{rem}

\begin{exas}\label{collapsenotunique}   Although the process of source elimination yields a graph $E_{sf}$  which is unique up to graph isomorphism, the process described in Remark \ref{collapserem}, if carried out in different orders,  does not necessarily yield isomorphic graphs.   For instance, let $E$ be the graph
$$ 
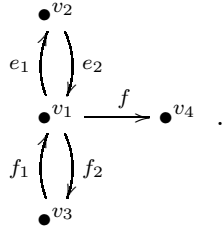
\xymatrix{  \bullet^{v_{2}} \ar@/^.5pc/[d]^{e_2} &   \\
\bullet^{v_1}\ar@/^.5pc/[d]^{f_2} \ar@/^.5pc/[u]^{e_1}  \ar[r]^f & \bullet^{v_4} \ \ .  \\
\bullet^{v_3} \ar@/^.5pc/[u]^{f_1}  & } \ \ $$
We note that $E = E_{sf}$, i.e., $E$ has no sources.

If we first collapse $E$ at $v_2$, and then subsequently at $v_3$, the resulting totally looped graph $F_2$  is

$$ 
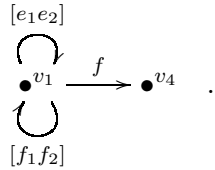
\xymatrix{  \bullet^{v_{1}} \ar@(ul,ur)^{[e_1e_2]} \ar@(dr,dl)^{[f_1f_2]} \ar[r]^f  &   \bullet^{v_4}  
 }  \ \ .$$

\bigskip

\noindent
On the other hand, if we instead  collapse $E$ at $v_1$, then the resulting totally looped graph $F_1$ is 

$$ 
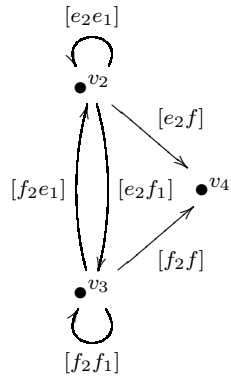
\xymatrix{  \bullet^{v_{2}} \ar@(ur,ul)_{[e_2e_1]}  \ar@/^.5pc/[dd]^{[e_2f_1]} \ar[dr]^{[e_2f]}  &   \\
 & \bullet^{v_4}  \\
\bullet^{v_3} \ar@/^.5pc/[uu]^{[f_2e_1]} \ar@(dr,dl)^{[f_2f_1]}  \ar[ur]_{[f_2f]}  & } $$

\bigskip
\bigskip
\noindent
Clearly $F_1$ and $F_2$ are not isomorphic as graphs.    (We note that both $F_1$ and $F_2$ are indeed totally looped; no loop is required at a sink, specifically, at $v_4$.)

\end{exas}

%{\bf Remark:   Nam, I am hopeful that we can simply include a reference which says that $L_K(E)$ is isomorphic to a direct sum of matrix rings of the correct size (corresponding to the isolated vertices in $E_{sf}$, together with $L_K(G)$ where $G$ contains no isolated vertices.  For now, I have moved the information about the hedgehog graph to the end of the document. }

%The following observation follows directly from Lemma \ref{disjunion}.  

%\begin{lem}\label{Moritawithisolated}  
%Suppose $v$ is an isolated vertex in a graph $E$, and write $E = \{v\} \sqcup G$ for the subgraph $G$ of $E$.  Then  $L_K(E) \cong K \oplus L_K(G)$.   In particular, $L_K(E)$ is Morita equivalent to $K \oplus L_K(G)$.    
%\end{lem}

We are now in position to establish the main result of this section.  

\begin{thm}\label{decompthm}   
Let $K$ be any field, and let $E$ be a finite graph.   Let $k$ denote the number of vertices of $E$ which are sinks in $E$, but which are not in   $E_{sf}^0$.  Let $F$ denote any totally looped graph which  is constructed  from the graph $E_{sf}$ via some step-by-step process, where at each step we collapse at a regular vertex  which is not the base of a loop. Then     
$$L_K(E) \ \mbox{is Morita equivalent to } (\prod^k_{i=1}K_i)\oplus L_K(F),$$
 where $K_i\cong K$ for $1\leq i \leq k$.  Consequently, there exist $k\geq 0$, and (when $k\geq 1$) positive integers $m_1, \dots, m_k$,  a positive integer $n$, and a full idempotent  $p \in M_n(L_K(F))$, for which  
$$ L_K(E) \cong (M_{m_1}(K)\oplus M_{m_2}(K)\oplus\cdots\oplus M_{m_k}(K))\oplus pM_n(L_K(F))p$$            
as $K$-algebras.
\end{thm}
\begin{proof}
By Remark \ref{isolatedverticesrem} we have that $L_K(E)$ is Morita equivalent to $L_K(V_k \sqcup E_{sf})$.  
%Since $L_K(\bullet) \cong K$, 
Applying $k$ times statements (1) and (2) of  Proposition \ref{Lpaproperties}  
%\ref{Moritawithisolated} 
gives that $L_K(E)$ is Morita equivalent to  $(\prod^k_{i=1}K_i)\oplus L_K(E_{sf})$.   Then one forms the sequence of graphs   $E_{sf} = F_0\rightarrow F_1\rightarrow\cdots\rightarrow\cdots
\rightarrow F_\ell := F,$ where each $F_{i+1}$ is produced from $F_i$ by collapsing at some regular vertex of $F_i$ which is not the base of a loop.    The Morita equivalence then follows directly from repeated application of Theorem \ref{collapsethm}.  

To establish the isomorphism, we use the    well-known fact  that for a field $K$, the only rings Morita equivalent to $K$ are of the form $M_n(K)$ for some positive integer $n$.   So the consequence follows immediately, using statements (1) and (2) of Theorem \ref{Moritaresults}. 
\end{proof}

%\begin{rem}
%Referring to the statement of Theorem \ref{decompthm}, it can be shown that the integer $k$ is precisely the integer $k$ described in Remark  \ref{isolatedverticesrem}.   
% is the integer described in  is the number of isolated vertices in $E_{sf}$, and, for each isolated vertex $v_i$ in $E_{sf}$,  $m_i$ is the number of paths in $E$ ending in $v_i$.   We discuss this below in Remark \ref{Namideas}.
%\end{rem}

%\begin{exas}
%
%[[ We should include Nam's examples here.  But maybe we can find also an example where the idempotent is not $p = 1$ ?? ]]
%
%
%\end{exas}    

%\begin{rem}\label{isoremark}
%
%[[   We can include some of Nam's analysis here.   But how much to include?   I would still very much like to not have to use the hedgehog construction, I think it simply makes the analysis too cumbersome.  The goal is to show that every corner of a Lpa is isomorphic to an Lpa, and to get that result we only need to have the isomorphism given in the above theorem.   But let's discuss that.]]
%
%\end{rem}  
	
\section{Corners of unital Leavitt path algebras}

The main goal of this section (indeed, of this article)  is to show that every corner of a Leavitt path algebra of a finite graph is also isomorphic to a Leavitt path algebra (Theorem \ref{cornerthm}).  Consequently, we achieve what on the surface seems to be a more general result (Theorem \ref{End(Q)general}):  for any finite graph $E$ and any nonzero finitely generated projective left $L_K(E)$-module $Q$, the endomorphism ring ${\rm End}_{L_K(E)}(Q)$ is isomorphic to $L_K(F)$ for some finite graph $F$.   

%Consequently, we get that a unital algebra which is Morita equivalent to the Leavitt path algebra of a finite  graph is also isomorphic to a Leavitt path algebra. In order to the main result we use Arklint-Ruiz's idea \cite[Section 4]{ar:cocka} for the proof of that every unital $C^*$-algebra which is stably isomorphic to a Cuntz-Krieger algebra is isomorphic to a Cuntz-Krieger algebra.

\begin{defn}[{\cite[Definition 2.2.21]{AAS}:  ``the restriction graph"}]\label{restdef}
Let $E$ be a graph and let $H$ be a hereditary subset of $E^0$.
We denote by $E_H$ the \textit{restriction graph}:

\begin{center}
$E^0_H : = H$, \quad $E^1_H := \{e\in E^1 \mid s(e) \in H\},$
\end{center}
and the source and range maps in $E_H$ are simply the source and range maps in $E$, restricted to $H$.  (We note that $H$ must be hereditary in order for the construction  $E_H$ to actually yield a graph, specifically, so that the restriction of the range function $r$ to edges having $s(e)\in H$ is defined.)  
\end{defn}

\begin{rem}\label{restrictioniscompleterem}
By construction, the restriction graph $E_H$ is a complete subgraph of $E$, so that by Proposition \ref{Lpaproperties}(3)  we may view $L_K(E_H)$ as a $K$-subalgebra of $L_K(E)$.  
\end{rem}

%%  Proposition about corners being isomorphic to restriction graphs for hereditary subsets:  moved to end of document.  
%For an idempotent $e$ in a ring $R$ and for any positive integer %$n$, we denote by $ne$ the matrix 

%$$eI_n = \left(\begin{tabular}{cccc}
%e&0&$\ldots$&0\\
%0&e&$\ldots$&0\\
%.&.&.&.\\
%0&0&$\ldots$&e\\
%\end{tabular}\right)\in M_n(R).$$

%The following result shows that the corner of the Leavitt path algebra of a finite graph in which every regular vertex is the base of a loop, which is arisen by a full idempotent, is again a Leavitt path algebra. %The following proposition is borrowed from \cite[Proposition 4.7]{ar:cocka}.

\begin{defn}(The ``strands of hair" extension of a graph) \label{hairdef}    Let $E$ be a finite graph, with $E^0 = \{v_1, v_2, \dots, v_t\}$.   Let $n_1, n_2, \dots, n_t$ be a sequence of positive integers.   We define the {\it strands of hair extension} (concisely: {\it hair extension}) graph $E^+(n_1, n_2, \dots, n_t)$  to be the graph $E$,  together with an extension by a ``strand of hair" of length $n_i-1$ at each $v_i$.   Graphically, $E^+(n_1, ... , n_t)$ is formed by adding these vertices and edges  to $E$:  
$$\xymatrix{  \bullet^{v_i^{n_i-1}} \ar[r]^{e_i^{n_i-1}} & \cdots \bullet^{v_i^2} \ar[r]^{e_i^2} & \bullet^{v_i^1} \ar[r]^{e_i^1} & }  \ $$
where  $r(e_i^1)=v_i$.    (So if $n_i=1$, then we attach no new edges at $v_i$.)    If the sequence $n_1, n_2, \dots, n_t$ is understood from context, we will denote  $E^+(n_1, n_2, \dots, n_t)$ simply by $E^+$. 
%We will denote $E^+(n_1, n_2, \dots, n_t)$ simply by $E^+$ when the sequence of integers is understood.  
\end{defn}

\begin{rem}\label{haircompleterem}
By construction,  $E$ is a complete subgraph of  any hair extension $E^+(n_1, n_2, \dots, n_t)$, so that by Proposition \ref{Lpaproperties}(3)  we may view $L_K(E)$ as a $K$-subalgebra of $L_K(E^+(n_1, n_2, \dots, n_t))$.  
\end{rem}

For clarification, we consider the following example.
\begin{exas}\label{hairex}
If $E$ is the graph
	
$$
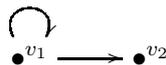
\xymatrix{\bullet^{v_1} \ar@(ul,ur)\ar[r]& \bullet^{v_2}}$$ then
$E^+(3,2)$ is the graph

$$
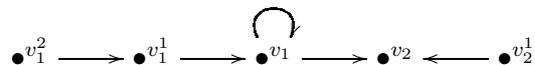
\xymatrix{\bullet^{v_1^2}\ar[r]& \bullet^{v_1^1}\ar[r]&
		\bullet^{v_1} \ar@(ul,ur)\ar[r]& \bullet^{v_2}&\bullet^{v_2^1}\ar[l]  }$$ 
\end{exas}

%For a left $R$-module $M$ and positive integer $t$, we denote the (external) direct sum of $t$ copies of $M$ by $tM$.  

The following Proposition illuminates exactly why the totally looped property should play such a central role in this analysis.  We first need a lemma.

\begin{lem}\label{everysubsetsaturated}({cf. \cite[Lemma 4.5]{ar:cocka}})
Let $E$ be totally looped. Then every subset of $E^0$ is saturated.   
%Consequently, for any nonempty subset $S\subseteq E^0$, $T(S)$ is hereditary and saturated.   
\end{lem}
\begin{proof} Let $H$ be a subset of $E^0$ and $v\in E^0$ a regular vertex with $r(s^{-1}(v)) \subseteq H$. By  hypothesis $v$ is the base of a loop $f$, i.e.,  $r(f) = v = s(f)$.  So $v\in r(s^{-1}(v))$, and so $v\in H$. Thus $H$ is saturated.  
\end{proof}

\begin{prop}  \label{hairprop}
Let $F = \{v_1, v_2, \dots, v_t\}$ be a totally looped finite graph.  Let $n_1, n_2, \dots, n_t$ be a sequence of positive integers, and let $E$ denote the hair extension  graph $F^+(n_1, n_2, \dots, n_t)$.      Let $Q$ be a nonzero finitely generated projective left $L_K(E)$-module.     Then there exist positive integers  $m_i $ ($1\leq i \leq u$),  and a hereditary subset  $T = \{ v_{j_1}, v_{j_2}, \dots, v_{j_u} \}$ of $F^0$ such that 
$$Q  \ \cong  \ \bigoplus_{i=1}^u \  m_i L_K(E)v_{j_i} .$$
Moreover, if $Q$ is a generator for $L_K(E)$-Mod, then $T = F^0$.  
\end{prop}

%\begin{defn} 

\begin{proof}
By Theorem \ref{AMPthm}(2)   we have that  
$$Q \cong \bigoplus_{w \in E^0} m''_w L_K(E) w$$  
for some (not necessarily unique) non-negative integers $m''_w$.      For any $w\in E^0$ which is of the form $v_i^j$ for some $j\geq 1$ (i.e., for each ``added" vertex $w$), we have $L_K(E) w = L_K(E)  v_i^j \cong L_K(E) v_i$ (by using  Theorem \ref{AMPthm}(1) $j-1$ times).    So by replacing appropriate summands, we have 

\begin{equation*}
Q \cong \bigoplus_{w \in F^0} m'_w L_K(E) w  \tag{\mbox{$\ast$}}
\end{equation*}
for some non-negative integers $m'_w$.  Clearly we can eliminate any summand for which $m'_w = 0$.  Denote by $T_1$  the set of remaining vertices (i.e., the set of vertices $w$ in $F^0$ for which $m'_w \geq 1$ in ($\ast$)).  So 
\begin{equation*}
Q \cong \bigoplus_{w \in T_1} m'_w L_K(E) w.  \tag{\mbox{$\ast \ast$}}
\end{equation*}

   Let $T$ denote the hereditary closure of $T_1$. (Note:  The hereditary closure of $T_1$ is the same regardless of whether we view $T_1\subseteq F^0$ or $T_1\subseteq E^0$; either way, $T\subseteq F^0$.)    We claim that 
$$Q \cong \bigoplus_{w \in T} m_w L_K(E) w,$$ 
where each $m_w \geq 1$.   For, let $z\in T$, let $v\in T_1$, and suppose that there is a path $p = e_1e_2 \cdots e_x$ from $v$ to $z$.   By an observation made in the Introduction,  we may assume that the sequence of vertices $v=s(e_1), s(e_2), \dots , s(e_x), r(e_x)=w$ which appear in $p$ contains no repeats.
%, for otherwise we could simply eliminate the appropriate closed subpaths at repeated vertices. 
 In particular, no $e_i$ is a loop.  
% and still have a path from $v$ to $z$.  
% no $e_i$ is a loop, as eliminating all loops from $p$ will still yield a path from $v$ to $z$.   
By Theorem \ref{AMPthm}(1)  we have $L_K(E)v \cong \oplus_{e\in s^{-1}(v)}L_K(E)r(e)$.   Since $v\in F^0$ we have that $r(e) \in F^0$ for all $e\in s^{-1}(v)$.   But because $F$ is totally looped, $v = r(f)$ for at least one loop $f \in s^{-1}(v)$; in addition, $f\neq e_1$ because $e_1$ is not a loop.  So this decomposition yields that 
\begin{equation*}
L_K(E)v  \ \cong  \ L_K(E)v \  \oplus L_K(E)r(e_1) \  \oplus \  \bigoplus_{e\in s^{-1}(v) \setminus \{f, e_1\}}L_K(E)r(e).   \tag{\mbox{$\dagger$}}
\end{equation*}
\noindent
  Now replace any one of the summands isomorphic to $L_K(E)v$ which appears in the  decomposition ($\ast$$\ast$)  of $Q$  by the isomorphic version of $L_K(E)v$ given in $(\dagger)$; note that such a replacement does not decrease the number of copies of the summand $L_K(E)v$ which appear in ($\ast$$\ast$).   Continuing this same process now on the summand $L_K(E)r(e_1)$, we see that after $x$ steps we arrive at a direct sum decomposition of $Q$ which includes a summand isomorphic to $L_K(E)z$, and which has not decreased the number of summands isomorphic to any given $L_K(E)w$ which appeared in  decomposition ($\ast$$\ast$) of $Q$. (Note that $L_K(E)z$ will appear as a summand of $Q$ in no more than $x$ steps, because there are no repeats in the sequence of vertices in $p$.)   This completes the proof of the claim, and establishes the displayed isomorphism of the statement.  
  
  For the second part, suppose that $Q$ is in addition a generator for $L_K(E)$-Mod.  Let $w\in F^0$.   Then for some positive integer $s$ there is a split epimorphism $sQ \to L_K(E)w \to 0$; so there are maps $\varphi \in {\rm Hom}_{L_K(E)}(sQ, L_K(E)w) $ and $\psi \in {\rm Hom}_{L_K(E)}(L_K(E)w, sQ)$ for which $\psi \varphi $ is the identity map on $L_K(E)w$.     But using Proposition \ref{End}(1) and the standard decomposition of maps to and from finite direct sums, the equation $w = (w)\psi\varphi$  yields elements $r_{i,j}$ and $r'_{i,j}$ in $L_K(E)$, with $1\leq i \leq u$ and $1 \leq j \leq s\cdot m_i$, for which 
  $$ w =  \sum_{i=1}^u \sum_{j=1}^{s\cdot m_i} w r_{i,j} v_i r'_{i,j} w .$$
  Because $w$ and each $v_i$ is in $F$, and because there are no paths which start in $F^0$ and end in any of the added vertices which produce $E$ as a hair extension of $F$, each expression $wr_{i,j} v_i$ and $v_i r'_{i,j} w$ is an element of $L_K(F)$.  Since $w$ and each $v_i$ an idempotent, we can therefore  view each term $wr_{i,j} v_i r'_{i,j}w$ in the displayed sum as an element of the ideal of $L_K(F)$ generated by the set of vertices $T\subseteq F^0$.  We have thus established that the ideal $I(T)$  of $L_K(F)$ generated by $T$ contains all vertices of $F^0$, and so $I(T) = L_K(F)$.  But $T$ is not only hereditary, it is by default saturated as well (Lemma \ref{everysubsetsaturated}).   We now apply \cite[Theorem 2.5.9]{AAS} to conclude that $T = F^0$.  
  %of the proposition as well. 
%  $         = \{ v_{j_1}, v_{j_2}, \dots, v_{j_u} \}$. 
\end{proof}

\begin{thm} \label{End(Q)hairthm}
Let $F$ be a finite totally looped graph, let $E = F^+$ be a hair  extension of $F$, and let $Q$ be a nonzero finitely generated projective left $L_K(E)$-module.    Then ${\rm End}_{L_K(E)}(Q)$ is isomorphic to a Leavitt path algebra.   Specifically, ${\rm End}_{L_K(E)}(Q) \cong L_K(G)$, where $G$ is a hair   extension of the restriction graph $F_T$ of $F$ by some hereditary subset $T$ of $F^0$.  (In particular, $G$ is a hair   extension of a totally looped graph.)    

Moreover, if $Q$ is in addition a generator for $L_K(E)$-Mod, then ${\rm End}_{L_K(E)}(Q) \cong L_K(G)$ where $G$ is a hair extension of $F$.
\end{thm}   

\begin{proof}
By Proposition \ref{hairprop}  we may decompose  $Q$ as $$Q \cong \bigoplus_{v\in T} m_v L_K(E)v,$$
where $T$ is a hereditary subset of $F^0$, and each $m_v \geq 1$.   Write $T = \{v_1, v_2, \dots , v_u\}$.  Note that there are $\sigma = \sum_{v\in T} m_v$ direct summands in the decomposition.      By  Proposition \ref{End}  
$ {\rm End}_{L_K(E)}(Q)$ is isomorphic to a $\sigma \times \sigma$ matrix ring, with entries described as follows.  The indicated matrices   may be viewed as consisting of rectangular blocks of size $m_{v_i} \times m_{v_j}$, where, for $1 \leq i \leq u$, $1\leq j \leq u$, the entries of the $(i,j)$ block are elements of the $K$-vector space   $v_i L_K(E) v_j$.    

On the other hand, because $T$ is hereditary, we may construct the restriction graph $F_T$ of $F$.  
%  Write  $E_T^0 = \{v_1, v_2, \dots , v_u\}$.    
Furthermore, because $m_i \geq 1$ for all $1\leq i \leq u$, we may construct  the hair   extension $G = F_T^+(m_1,m_2, \dots , m_u)$ of $F_T$.    For each $1 \leq i \leq u$, and each $1\leq y \leq m_i$,  let $p_i^y:= e_i^y \cdots e_i^1$ denote the (unique) path in $G = F_T^+$ having  $s(p_i^y) = v_i^y$, and $r(p_i^y) = v_i$.    Note that, because of the specific configuration of the added vertices and edges used to build $G$ as $F_T^+$, repeated application of Remark \ref{oneedge} gives that  $p_i^y (p_i^y)^* = v_i^y$ in $L_K(G)$.      Note also that $|G^0| =   \sum_{1\leq i \leq u} m_i$, which  is precisely $\sigma$.      Writing $L_K(G) = \oplus_{v\in G^0} L_K(G)v$ and again using Proposition \ref{End}, we get that 
$L_K(G)$ is isomorphic to the $\sigma \times \sigma$ matrix ring with entries described as follows.   For $1 \leq i, j \leq u$, and $1 \leq y \leq m_i$, $1 \leq z \leq m_j,$ the entries in the row indexed by $(m_i, y)$ and column indexed by $(m_j, z)$ are elements of $v_i^y L_K(G) v_j^z.$   

We now show that these two $\sigma \times \sigma$ matrix rings are isomorphic as $K$-algebras.  To do so, we show first that for each pair  $(m_i, y)$,  $(m_j, z)$ with $1 \leq i, j \leq u$, and $1 \leq y \leq m_i$, $1 \leq z \leq m_j$, there is a $K$-vector space isomorphism
$$\varphi = \varphi_{(m_i, y), (m_j,z)} :   v_i L_K(E) v_j \rightarrow v_i^y L_K(G) v_j^z.$$ 
      For $r
      % = r_{(m_i, y), (m_j,z)} 
      \in L_K(E)$ we define
$$ \varphi_{(m_i, y), (m_j,z)}(v_i r v_j) = p_i^yv_i r v_j (p_j^z)^*.$$
Writing $r$ as a sum of elements of the form $k \alpha \beta^*$ with $k\in K$ and $\alpha, \beta \in {\rm Path}(E)$, we have that  $v_i r v_j$ may be viewed as a sum of elements $k\alpha \beta^*$ with $s(\alpha) = v_i$ and $s(\beta)= v_j$.  Because any path which starts in $T$ must have all of its vertices in $T$ (because $T$ is hereditary, and there is no path from $T \subseteq F^0$ to any of the added vertices which yield $G$ as $F_T^+$), we have that the expression $v_i r v_j$ is indeed an element of $L_K(G)$, which in turn yields that $ p_i^yv_i r v_j (p_j^z)^* \in v_i^y L_K(G) v_j^z$.   

 That $\varphi$ is $K$-linear is clear.   Further, $\varphi$ is a monomorphism:  if $p_i^yv_i r v_j (p_j^z)^* = 0$ then multiplying on the left by $(p_i^y)^*$ and on the right by $p_j^z$ yields $v_i r v_j = 0$.   To show  $\varphi$ is surjective:  for $v_i^y s v_j^z \in v_i^y L_K(G) v_j^z$ with $s\in L_K(G)$, define $s' = (p_i^y)^*v_i^y  s v_j^z p_j^z \in v_i L_K(G) v_j$.  But  using the fact that there are no paths from elements of $T$ to any of the newly added vertices which yield $G$ as $(F_T)^+$, we have as above that $s'$ may be viewed as an element of $L_K(E)$.    Then, using the previous observation that $p_i^y (p_i^y)^* = v_i^y$ in $L_K(G)$, we conclude that $\varphi(s') = p_i^y (p_i^y)^* v_i^y s v_j^z p_j^z (p_j^z)^*  = v_i ^y s v_j^z$, and thus $\varphi$ is surjective.

We now define $\Phi$ to be the $K$-space isomorphism between the two matrix rings induced by applying each of the $\varphi_{(m_i, y), (m_j,z)}$ componentwise.   We need only show that these componentwise isomorphisms respect the corresponding matrix multiplications.    But to do so, it suffices to show that the maps behave correctly in each component.    That is, we need only show, for each $m_\ell$ ($1\leq \ell \leq u$)   and each $x$ ($1 \leq x \leq m_\ell$), that 
$$ \varphi_{(m_i, y), (m_\ell,x)}(v_i r v_\ell) \cdot  \varphi_{(m_\ell, x), (m_j,z)}(v_\ell  r' v_j) =  \varphi_{(m_i, y), (m_j,z)}(v_i r v_\ell r'  v_j).$$ 
But this is immediate, as $(p_\ell^x)^* p_\ell^x = v_\ell$ for each $v_\ell \in T$ and $1 \leq x \leq m_\ell.$  

The additional statement follows from the final assertion of Proposition \ref{hairprop}.  
% the entries in the row indexed by $(m_i, y)$ and column indexed by $(m_j, z)$ are elements of $v_i^y L_K(G) v_j^z.$   the individual en need to show that the
\end{proof}

\begin{rem} In the previous proof, $E$ is an arbitrary hair   extension of $F$, by some sequence of $|F^0|$ integers.  As well,  $G$ is a hair   extension of  a subgraph $F_T$ of $F$, by some sequence of  $|F_T^0|$ integers.    In general there need be no relationship whatsoever between the two sequences of integers.  
\end{rem}

The following example will help illuminate the ideas of Theorem \ref{End(Q)hairthm}.  

\begin{exas}\label{isoexample}
Let $F$ be  graph

\medskip

$$F \ \ \ = \ \ \  \xymatrix{  \bullet^{v_{1}}\ar@(ul,ur) \ar[r]  &  \bullet^{v_2}  \ar@(l,u) \ar@(r,u) \ar@/^.5pc/[l] \ar[r]& \bullet^{v_{3}}  & \bullet^{v_4}\ar@(ul,ur) \ar[l] } \ \ . $$
\noindent
Then $F$ is totally looped (note that $v_3$ is a sink, so no loop is required at $v_3$).  Let $E$ be the hair   extension $F^+(3,1,2,3)$ of $F$, pictured here: 

$$
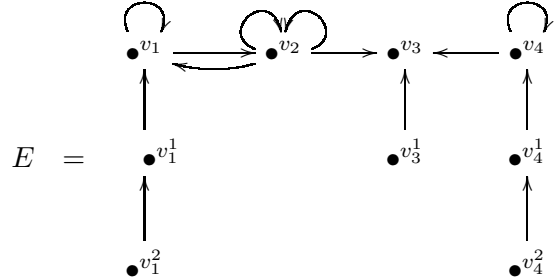
\xymatrix{  \bullet^{v_{1}}\ar@(ul,ur) \ar[r]  &  \bullet^{v_2}  \ar@(l,u) \ar@(r,u) \ar@/^.5pc/[l] \ar[r]& \bullet^{v_{3}}  & \bullet^{v_4}\ar@(ul,ur) \ar[l] \\
\hspace{-.5in} E \ \    = \ \ \ \ \  \bullet^{v_1^1} \ar[u] & & \bullet^{v_3^1} \ar[u] & \bullet^{v_4^1} \ar[u]  \\
\bullet^{v_1^2}  \ar[u] & & & \bullet^{v_4^2} \ar[u]  } $$
Let $R$ denote $L_K(E)$.  Consider the (arbitrarily chosen) nonzero finitely generated projective left $R$-module $Q = Rv_1^2$.   We write $Q$ in the form indicated in Proposition \ref{hairprop}, as follows.  Using the isomorphism of Theorem  \ref{AMPthm}(1) multiple times, we have
%\begin{align*}
$$Q   \cong Rv^1_1\cong Rv_1 \cong Rv_1 \oplus Rv_2 \cong Rv_1 \oplus (Rv_1 \oplus 2Rv_2 \oplus Rv_3) \cong 2Rv_1 \oplus 2 Rv_2 \oplus Rv_3.$$
  %     & \cong 2Rv_1 \oplus (Rv_1 \oplus 2Rv_2 \oplus Rv_3) \cong 3Rv_1 \oplus 2Rv_2 \oplus Rv_3 . 
%\end{align*}

%$$Q \cong Rv_1^1 \cong Rv_1 \cong Rv_1 \oplus Rv_2 \cong Rv_1 \oplus (Rv_1 \oplus Rv_2) \cong 2Rv_1 \oplus (Rv_1 \oplus 2Rv_2 \oplus Rv_3) \cong 3Rv_1 \oplus 2Rv_2 \oplus Rv_3.$$
The hereditary subset of $F^0$ corresponding to $Q$ is $T = \{v_1, v_2, v_3\}$, so

$$F_T \ \   \ =  \ \   \ 
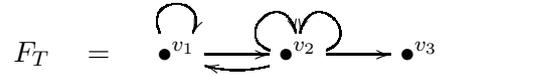
\xymatrix{  \bullet^{v_{1}}\ar@(ul,ur) \ar[r]  &  \bullet^{v_2}  \ar@(l,u) \ar@(r,u) \ar@/^.5pc/[l] \ar[r]& \bullet^{v_{3}} & }   .$$
The decomposition $$Q \cong 2Rv_1 \oplus 2Rv_2 \oplus Rv_3$$ dictates that we construct the hair   extension $G = F_T^+(2,2,1)$ of $F_T$, graphically,

\medskip

$$G   \ \   \  =   \ \   \  
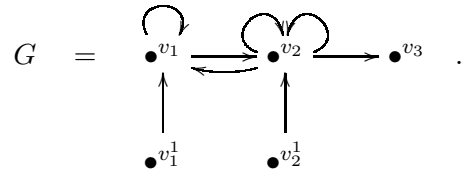
\xymatrix{  \bullet^{v_{1}}\ar@(ul,ur) \ar[r]  &  \bullet^{v_2}  \ar@(l,u) \ar@(r,u) \ar@/^.5pc/[l] \ar[r]& \bullet^{v_{3}}  \\
\bullet^{v_1^1} \ar[u] &\bullet^{v_2^1} \ar[u] &   
 }   \ \   . $$
By Proposition \ref{hairprop} we have  ${\rm End}_{L_K(E)}(Q) \cong L_K(G)$.   For notational simplification, let $S$ denote $L_K(G)$.  Then   the explicit description of these two algebras as matrix rings as described in the proof of  Proposition \ref{hairprop} is:
\medskip
\footnotesize
$$   \left(\begin{tabular}{ccccc} 
$v_1Rv_1$&$v_1Rv^1_1$&$v_1Rv_2$&$v_1Rv_2^1$&$v_1Rv_3$\\
$v_1^1Rv_1$&$v_1^1Rv^1_1$&$v_1^1Rv_2$&$v_1^1Rv_2^1$&$v_1^1Rv_3$\\
$v_2Rv_1$&$v_2Rv^1_1$&$v_2Rv_2$&$v_2Rv_2^1$&$v_2Rv_3$\\
$v_2^1Rv_1$&$v_2^1Rv^1_1$&$v_2^1Rv_2$&$v_2^1Rv_2^1$&$v_2^1Rv_3$\\
$v_3Rv_1$&$v_3Rv^1_1$&$v_3Rv_2$&$v_3Rv_2^1$&$v_3Rv_3$\\
\end{tabular}\right) 
\cong
 \left(\begin{tabular}{ccccc} 
$v_1Sv_1$&$v_1Sv_1$&$v_1Sv_2$&$v_1Sv_2$&$v_1Sv_3$\\
$v_1Sv_1$&$v_1Sv_1$&$v_1Sv_2$&$v_1Sv_2$&$v_1Sv_3$\\
$v_2Sv_1$&$v_2Sv_1$&$v_2Sv_2$&$v_2Sv_2$&$v_2Sv_3$\\
$v_2Sv_1$&$v_2Sv_1$&$v_2Sv_2$&$v_2Sv_2$&$v_2Sv_3$\\
$v_3Sv_1$&$v_3Sv_1$&$v_3Sv_2$&$v_3Sv_2$&$v_3Sv_3$
\end{tabular}\right)  .$$
\smallskip

\normalsize

%\ar@/^.5pc/[r] 
\end{exas}

\begin{cor}\label{haircor}   Let $E$ be a graph which arises as a hair   extension of a totally looped finite graph $F$.   Let $\varepsilon$ be a nonzero idempotent in $L_K(E)$.  Then the corner algebra $\varepsilon L_K(E) \varepsilon$ is isomorphic to a Leavitt path algebra. 
 More specifically, $\varepsilon L_K(E) \varepsilon$  is isomorphic to the Leavitt path algebra of a hair extension of $F_T$, 
 % Leavitt path algebra of a graph which arises as a hair   extension of a totally looped finite graph arising 
 where   $F_T$ is the (totally looped) restriction graph of $F$ to some hereditary subset $T$ of $F^0$.

Moreover, in case $\varepsilon$ is a full idempotent in $L_K(E)$, then $\varepsilon L_K(E) \varepsilon$ is isomorphic the Leavitt path algebra of a hair extension of $F$ itself.   
\end{cor}

\begin{proof}
The first statement follows directly from Theorem \ref{End(Q)hairthm}, as $L_K(E)\varepsilon$ is a nonzero finitely generated projective left $L_K(E)$-module, and ${\rm End}_{L_K(E)}(L_K(E)\varepsilon) \cong \varepsilon L_K(E) \varepsilon.$    The statement about full idempotents follows from Proposition \ref{hairprop}, because an idempotent $f$ in a ring $R$ is full  precisely when $Rf$ is a generator for $R$-Mod.  
\end{proof}

There is a specific hair extension  construction for arbitrary graphs which is already known, and which will be useful in establishing our main result.

\begin{defn}[{\cite[Definition 9.1]{at:iameoga}}]\label{ATdef}
For any finite graph $E$ and  positive integer $n$, let $M_nE$ denote the hair extension  graph 
$$M_nE = E^+(n,n,\dots,n).$$ 
In other words, $M_nE$ is constructed from $E$ by attaching a strand of hair of length $n-1$ of the form
$$\xymatrix{  \bullet^{v^{n-1}} \ar[r] & \cdots \bullet^{v^2} \ar[r] & \bullet^{v^{1}} \ar[r] &   } $$ to   each $v\in E^0$.
\end{defn}
%For clarification, we consider the following example.
%\begin{exas}
%If $E$ is the graph
%	
%$$\xymatrix{\bullet\ar@(ul,ur)\ar[r]& \bullet}$$ then
%$M_3E$ is the graph
%
%$$\xymatrix{\bullet\ar[r]& \bullet\ar[r]&
	%	\bullet \ar@(ul,ur)\ar[r]& \bullet&\bullet\ar[l]&\bullet\ar[l]}$$ 
%\end{exas}

\begin{prop}[{\cite[Proposition 9.3]{at:iameoga}}]\label{ATprop}
Let $K$ be any field, $E$ a finite graph and $n$ a positive integer. Then
there exists a $K$-algebra isomorphism $$\varphi: M_n(L_K(E))\longrightarrow L_K(M_nE).$$   In particular, any full $n\times n$ matrix ring over a Leavitt path algebra is isomorphic to the Leavitt path algebra of a hair   extension of $E$.  
% such that $\phi(vE_{11}) = v$ for all $v\in E^0$, where $E_{ij}$ are the matrix units of $M_n(L_K(E))$.	
\end{prop}

 For each positive integer $n$ we denote by $A_n$ the ``straight line graph" having $n$ vertices and $n-1$ edges:
$$A_n \ \ = \ \ \xymatrix{  \bullet^{v_{n-1}} \ar[r] & \bullet^{v_{n-2}} \ar[r] & \cdots \bullet^{v_2} \ar[r] & \bullet^{v_{1}} \ar[r] & \bullet^{v_0}  }. $$
 
 Easily (or, applying Proposition \ref{ATprop}  to $E_{triv}$) we get 
 
 \begin{lem}\label{AnLemma}  For any positive integer $n$, $L_K(A_n) \cong M_n(K)$. 
 \end{lem}
 
We are finally in position to achieve the main result of this article.   

\begin{thm}\label{cornerthm}
Let $K$ be any field.   Let $E$ be any finite graph, and let $\varep$ be any nonzero idempotent in $L_K(E)$.   Then the corner $\varep L_K(E) \varep$ of $L_K(E)$  is isomorphic to a Leavitt path algebra.  
\end{thm}

\begin{proof}
By Theorem \ref{decompthm} we have   
$$L_K(E) \cong (M_{m_1}(K)\oplus M_{m_2}(K)\oplus\cdots\oplus M_{m_k}(K))\oplus pM_n(L_K(F))p$$
for some $k\geq 0$, integers $m_1, \dots , m_k$, and some  full idempotent $p$ in $M_n(L_K(F))$ for some positive integer $n$ and totally looped graph $F$.   By Proposition \ref{ATprop}  $M_n(L_K(F)) \cong L_K(M_nF)$.   Let $\gamma$ denote the image of $p$ under this isomorphism; so $\gamma$ is a full idempotent in $L_K(M_nF)$.  Then as $K$-algebras we have 
$$L_K(E) \cong (M_{m_1}(K)\oplus M_{m_2}(K)\oplus\cdots\oplus M_{m_k}(K))\oplus \gamma L_K(M_nF) \gamma.$$
Since $M_nF = F^+(n,n, \dots , n)$ is a hair   extension of the totally looped graph $F$, Corollary \ref{haircor}  yields that $\gamma L_K(M_nF) \gamma \cong L_K(G_1)$ for some finite graph $G_1$, where $G_1$ is a hair   extension of $F$.   So we get
$$L_K(E) \cong (M_{m_1}(K)\oplus M_{m_2}(K)\oplus\cdots\oplus M_{m_k}(K))\oplus L_K(G_1);$$
denote this $K$-algebra isomorphism by $\Phi$.   
Write  $\Phi(\varep) = (\epsilon_1, \epsilon_2, \dots , \epsilon_k, \epsilon)$; then each of $\epsilon$ and $\epsilon_i$ ($1\leq i \leq k$) is an idempotent.  Reordering if necessary, we may eliminate any summand for which $\epsilon_i = 0$,  and thereby get 
$$\varep L_K(E) \varep \cong (\epsilon_1 M_{m_1}(K)\epsilon_1 \oplus \epsilon_2 M_{m_2}(K)\epsilon_2 \oplus\cdots\oplus \epsilon_\ell M_{m_\ell}(K)\epsilon_\ell)\oplus \epsilon L_K(G_1) \epsilon$$
for some $\ell \leq k$.     If $\epsilon = 0$ then we eliminate the summand $\epsilon L_K(G_1)\epsilon$; otherwise,    again invoking Corollary \ref{haircor}, we have that  $ \epsilon L_K(G_1) \epsilon \cong L_K(G)$ for some graph $G$.  
Thus we have 
$$\varep L_K(E) \varep \cong (\epsilon_1 M_{m_1}(K)\epsilon_1 \oplus \epsilon_2 M_{m_2}(K)\epsilon_2 \oplus\cdots\oplus \epsilon_\ell M_{m_\ell}(K)\epsilon_\ell)\oplus L_K(G).  $$
It is well-known that any corner of a full matrix ring over a field $K$ is isomorphic to a full matrix ring (possibly of smaller size) over $K$.   So 
$$\varep L_K(E) \varep \cong (M_{n_1}(K) \oplus M_{n_2}(K) \oplus\cdots\oplus M_{n_\ell}(K))\oplus L_K(G) $$
for some integers $1 \leq n_i \leq m_i$ ($1\leq i \leq \ell$).   
Since $M_t(K) \cong L_K(A_t)$ for any positive integer $t$ (Lemma \ref{AnLemma}), this last isomorphism with Proposition \ref{Lpaproperties}(2) yields 
$$ \varep L_K(E) \varep \cong L_K( A_{n_1} \sqcup A_{n_2} \sqcup \cdots \sqcup A_{n_\ell} \sqcup G ),$$
thus establishing the result.
\end{proof}

\begin{cor}\label{Moritacor}
Let $K$ be any field.  Let $A$ be a $K$-algebra which is Morita equivalent to a Leavitt path algebra.  Then $A$ is isomorphic to a Leavitt path algebra.
\end{cor} 

\begin{proof}
If $A$ is Morita equivalent to $L_K(E)$, then (see Theorem \ref{Moritaresults}(1)) there exists a positive integer $n$ and a (full) idempotent $p\in M_n(L_K(E))$ for which $A \cong p M_n(L_K(E)) p$.  But using Proposition \ref{ATprop}, we have $p M_n(L_K(E)) p \cong \varep L_K(M_nE) \varep$ for some idempotent $\varep \in L_K(M_nE)$.      Finally, we invoke Theorem \ref{cornerthm}  to get the desired result.  
\end{proof}

Although the following result seems on the surface to be a generalization of Theorem \ref{cornerthm}, this result in fact follows as a consequence of Theorem \ref{cornerthm}.  
\begin{thm}\label{End(Q)general}
Let $K$ be any field.  Let $E$ be any finite graph, and let $Q$ be a nonzero finitely generated projective left $L_K(E)$-module.   Then ${\rm End}_{L_K(E)}(Q)$ is isomorphic to a Leavitt path algebra.  
\end{thm}

\begin{proof}
   Suppose $Q$ is generated by $n$ elements as an $L_K(E)$-module.   Then by Theorem \ref{Moritaresults}(3), under the standard Morita equivalence $\Psi$ between $L_K(E)-Mod$ and $M_n(L_K(E))-Mod $,  $\Psi (Q) $ is isomorphic to a direct summand of $M_n(L_K(E))$, i.e., $\Psi (Q) \cong M_n(L_K(E))q $ for some idempotent $q\in M_n(L_K(E))$.   But the equivalence yields that $${\rm End}_{L_K(E)}(Q) \cong {\rm End}_{M_n(L_K(E))}(\Psi (Q) ) \cong {\rm End}_{M_n(L_K(E))}( M_n(L_K(E))q).$$   This in turn by Proposition \ref{ATprop}   is isomorphic to ${\rm End}_{L_K(M_nE)}( L_K(M_nE) \varep) $ for some idempotent $\varep \in L_K(M_nE)$, which in turn is isomorphic to $ \varep L_K(M_nE) \varep$.  Now Theorem \ref{cornerthm}  gives the result.   
\end{proof}

We close this article with a series of remarks.

\begin{rem}\label{Cstarrem}
There is a tight but not fully understood connection between results about Leavitt path algebras and results about their graph $C^*$-algebra analogs.   The connection continues in this context as well.  Specifically, Arklint and Ruiz \cite{ar:cocka},  and Arklint, Gabe and Ruiz  \cite{ArklintGabeRuiz}  have established (among many other things) that for a finite graph $E$, any corner $pC^*(E)p$ of the graph $C^*$-algebra $C^*(E)$ by a projection $p$ is isomorphic to a graph $C^*$-algebra.   
%The tools utilized in their approach are significantly different from the tools we use herein.  
\end{rem}

\begin{rem}\label{allidempotentsrem}
Our ability to establish Theorem \ref{cornerthm} for {\it all} nonzero idempotents $\varep$ in $L_K(E)$ may seem surprising.  Specifically, we need not assume that $\varep$ possess any additional properties (e.g., that $\varep$ be full, or that $\varep \in L_K(E)_0$, or that $\varep = \varep^*$).  The point here is that by the theorem of  Ara,  Moreno and Pardo (Theorem \ref{AMPthm}(2)), we have a description up to isomorphism of any left $L_K(E)$-module of the form $L_K(E)\varep$ for {\it all} idempotents $\varepsilon$ of $L_K(E)$ in terms of idempotents of the form $L_K(E)v$ where $v\in E^0$.    The foundational result which is used to establish Theorem \ref{AMPthm} is the fact that the Leavitt path algebra $L_K(E)$ may be viewed as the ``Bergman algebra" of the monoid $M_E$, which allows the extremely powerful \cite[Theorem 6.2]{Bergman} to be invoked. 
\end{rem}

\begin{rem}\label{Namideas}
By combining the germane ideas in the proofs of Theorem \ref{cornerthm} and  Corollary \ref{Moritacor}, we can in fact establish a more precise result about algebras which are Morita equivalent to Leavitt path algebras.    Specifically, 
%We end this section and the paper by showing that every unital algebra that is Morita equivalent to a Leavitt path algebra is again a Leavitt path algebra.
%
%\begin{thm}
let $K$ be any field, $A$ a unital $K$-algebra, $E$ a finite graph, and $F$ a finite graph in which every regular vertex is the base of a loop that is obtained from the graph $E_{sf}$ via some step-by-step process of  collapsing at a regular vertex which is not the base of a loop. Then $A$ is Morita equivalent to $L_K(E)$ if and only if 

%\vspace{-.15in}
\begin{itemize}
\item[(1)] there exists a finite acyclic graph $V$ such that the number of  sinks in $V$ is equal to the number of all those sinks of $E$  which are not in $E^0_{sf}$, 
\item[(2)] there exist a positive integer $k$, and a hereditary subset $H$ of $(M_kF)^0$ containing $F^0$, and 
\item[(3)]  $A$ is isomorphic to $L_K(V \sqcup (M_kF)_{H})$.
\end{itemize}

\end{rem}

%%%%%%%%%%%%%%%%%%%%%%%%%%%%%%%%%%
%
%
%%%%     BIBLIOGRAPHY
%
%
%%%%%%%%%%%%%%%%%%%%%%%%%%%%%%%%%%

\vskip 0.5 cm \vskip 0.5cm {

\end{document}
\begin{thebibliography}{99}

\bibitem{a:lpatfd} G. Abrams, Leavitt path algebras: the first decade, \textit{Bull. Math. Sci.} \textbf{5} (2015), 59-120.

\bibitem{ap:tlpaoag05} G.~Abrams and G.~Aranda Pino, The Leavitt path algebra of a graph, \emph{Journal of Algebra}, \textbf{293} (2005), 319--334.

%\bibitem{apm:fdlpa} G.~Abrams, G.~Aranda Pino and M. Siles Molina, Finite-dimensional Leavitt path algebras, \emph{Journal of Pure and Applied Algebra} {\bf 209} (2007), 753--762.

%\bibitem{apm:lflpa}
%G. Abrams, G. Aranda Pino and M. Siles Molina, Locally finite
%Leavitt path algebras, \emph{Israel J. Math}, \textbf{165} (2008), 329-348.

\bibitem{aalp:tcqflpa}
G. Abrams, P. N. \'{A}nh, A. Louly, and E. Pardo, The classification question for Leavitt path algebras,
\emph{J. Algebra}, \textbf{320} (2008), 1983--2026.

\bibitem{AAS} G. Abrams, P. Ara, and M. Siles Molina, \emph{Leavitt path
	algebras}.   Lecture Notes in Mathematics series, Vol. 2191, Springer-Verlag Inc., 2017.

\bibitem{alps:fiitcolpa}
G. Abrams, A. Louly, E. Pardo, and C. Smith, Flow invariants in the classification of Leavitt path algebras,
\emph{J. Algebra}, \textbf{333} (2011), 202--231.


\bibitem{anp:lpahugn}
G. Abrams, T. G. Nam, and N. T. Phuc, Leavitt path algebras having Unbounded Generating Number,
\emph{Journal of Pure and Applied Algebra} {\bf 221} (2017), 1322-1343.

\bibitem{at:iameoga}
G. Abrams and M. Tomforde, Isomorphism and Morita equivalence of graph algebras, \emph{Trans. Amer.\ Math.\ Soc.}, \textbf{363} (2011), 3733--3767.

%\bibitem{aajz:solpaopg}
%A. Alahmedi, H. Alsulami, S. Jain and Efim I. Zelmanov, Structure of Leavitt path algebras of polynomial growth, \emph{Proc. Natl. Acad. Sci. USA} \textbf{110} (2013), 15222-15224.

%\bibitem{ag:lpaosg} P.~Ara and K. Goodearl, Leavitt path algebras of separated graphs, \emph{J. Reine Angew. Math.}, \textbf{669} (2012), 165--224.

\bibitem{AF}  F. Anderson and K. Fuller, \emph{Rings and Categories of Modules}, Second Edition.  Graduate Texts in Mathematics series, Vol. 13, Springer-Verlag Inc., 1992. 

\bibitem{amp:nktfga} P.~Ara, M.\,A.~Moreno, and E.~Pardo, Nonstable
K-theory for graph algebras, \emph{Algebr.\ Represent.\ Theory},
\textbf{10} (2007), 157--178.

\bibitem{ar:fpsmolpa} P. Ara and K.M. Rangaswamy, Finitely presented simple
modules over Leavitt path algebras, \emph{Journal of Algebra},
\textbf{417} (2014), 333--352.


%\bibitem{al:nlpaatra} G. Aranda Pino and L. Va\v{s}, Noetherian Leavitt
%path algebras and their regular algebras, \emph{Mediterr. J. Math.},
%\textbf{10} (2013), 1633--1656.

%\bibitem{bass:spicakiakt}
%H. Bass, \emph{Some prolems in ``classical" algebraic K-theory, in
%``Algebraic K-Theory II"}, Lecture Notes in Math., Vol. 342, pp. 3--73,
%Springer-Verlag, New York-Berlin, 1972.

%\bibitem{bh:ctopmoatraipe}
%S. M. Bhatwadekar, Cancellation theorems of projective modules over
%a two-dimension ring and its polynomial extensions,
%\emph{Compositio Mathematica}, \textbf{128} (2001), 339--359.

\bibitem{ar:cocka}
S. Arklint and E. Ruiz, Corners of Cuntz-Krieger algebras, \emph{Trans.
Amer.\ Math.\ Soc.}, \textbf{367} (2015), 7595--7612.

\bibitem{ArklintGabeRuiz}    S. Arklint, J. Gabe, and E. Ruiz, \emph{Hereditary $C^*$-subalgebras of graph $C^*$-algebras},   arXiv:   1604.03085v2.  


\bibitem{Bergman}  G. M. Bergman,  
  Coproducts and some universal ring constructions,  
 \emph{Trans. Amer. Math. Soc.}  {\bf 200}  (1974), 33--88 

%\bibitem{c:irolpa}
%X.W. Chen, Irreducible representations of Leavitt path algebras, \emph{Forum Math.} \textbf{27} (2015), 549--574.


%\bibitem{cy:hclpaagpm}
%X.W. Chen and D. Yang, Homotopy categories, Leavitt path algebras and Gorenstein projective modules, \emph{Inter. Math. Res. Not.} \textbf{10} (2015), 2597--2633.

%\bibitem{cfst:aggolpa}
%L.O. Clark, C. Farthing, A. Sims, and M. Tomforde, A groupoid generalisation of Leavitt path algebras, \emph{Semigroup Forum} 
%\textbf{89} (2014), 501--517. 


%\bibitem{c:rfor}
%P. M. Cohn, Rank functions on rings, \emph{Journal of Algebra},
%\textbf{133} (1990), 373--385.

%\bibitem{c:fhrtsd}
%P. M. Cohn, From Hermite rings to Sylvester domains, \emph{Proc.
%Amer. Math. Soc}, \textbf{128} (2000), 1899--1904.

%\bibitem{c:acfartbp}
%P. M. Cohn, Another criterion for a ring to be projective-free,
%\emph{Bull. London Math. Soc}, \textbf{37} (2005), 857--859.


%\bibitem{col:tsiilpa} P.~Colak, Two-sided ideals in Leavitt path algebras, \textit{J.~Algebra Appl.}, \textbf{10} (2011), 801--809.
  
\bibitem{cr06:csameogca} T. Crisp and D. Gow, Contractible subgraphs and Morita equivalence of graph $C^*$-algebras, \textit{Proc. Amer. Math. Soc.} \textbf{134} (2006), no. 7, 2003--2013. 
  
%\bibitem{h:tdolpa}
%R. Hazrat, The dynamics of Leavitt path algebras, \emph{J. Algebra} \textbf{384} (2013), 242–266.  

%\bibitem{g:srrarf} K. R. Goodearl, Simple regular rings and rank functions,
%\emph{Math. Ann.}, \textbf{214} (1975), 267 - 278.

%\bibitem{gh:rfakorr} K. R. Goodearl and D. Handelman, Rank functions and
%$K_0$ of regular rings, \emph{J. Pure Appl. Algebra}, \textbf{7} (1976), 195 - 216.

%\bibitem{lam:afcinr} T. Y. Lam, \textit{A First Course in Noncommutative Rings}, 2nd ed. Springer-Verlag, New York-Berlin, 2001.

%\bibitem{l:sc}
%T. Y. Lam, \emph{Serre's conjecture}, Lecture Notes in Math., Vol. 635,
%Springer-Verlag, New York-Berlin, 1978.

\bibitem{Lam}   T.Y. Lam, {\it Lectures on Modules and Rings}.  Graduate Texts in Mathematics series, Vol. 189, Springer-Verlag Inc., 1999. 


%\bibitem{leav:mwibn}
%W.G. Leavitt, Modules without invariant basis number, \emph{Proc. Amer. Math.
%Soc.},  \textbf{8} (1957), 322 --328.

\bibitem{leav:tmtoar}
W. Leavitt, The module type of a ring, \emph{Trans.
Amer.\ Math.\ Soc.} \textbf{42} (1962), 113--130.

%\bibitem{l:tilc}
%H. Li, The injective Leavitt complex, \emph{Algebr. Represent. Theory} \textbf{21} (2018), 833--858.


%\bibitem{leav:tmtohi}
%W.\,G.~Leavitt, The module type of homomorphic images,
%\emph{Duke Math.\ J.}, \textbf{32} (1965), 305--311.

%\bibitem{m:itakt}
%J. W. Milnor, \emph{Introduction to Algebraic K-theory}, Ann. Math. Studies
%No. 72. Princeton University Press, Princeton, NJ 1971.

%\bibitem{p:tipfhtg} E. Pardo, The isomorphism problem for Higman-Thompson groups, \emph{Journal of Algebra} \textbf{344} (2011), 172--183. 

\bibitem{r:ga}
I. Raeburn, \emph{Graph Algebras}.   CBMS Regional Conference Series in Mathematics, Vol. 103,
American Mathematical Society, Providence, RI, 2005, vi+113 pp. Published for the Conference
Board of the Mathematical Sciences, Washington, DC.

%\bibitem{swan:pmolpr}
%R.\,G.~Swan, Projective modules over Laurent polynomial rings, \emph{Trans.
%Amer.\ Math.\ Soc.}, \textbf{237} (1978), 111--120.

%\bibitem{s:ceigmopaoq}
%S.\,P.~Smith, Category equivalences involving graded modules over path algebras of quivers, \emph{Adv. Math.} \textbf{230} (2012), 1780--1810.


\bibitem{sor:gcosga} A. P. W. S\o{}rensen, Geometric classification of simple graph algebras, \emph{Ergodic Theory and Dynamical Systems}, \textbf{33} (2013), 1199--1220.

%\bibitem{stein:agatdisa}
%B. Steinberg, A groupoid approach to discrete inverse semigroup algebras, \emph{Adv. Math.} \textbf{223} (2010), 689--727.

\bibitem{tomf:utaisflpa} M. Tomforde, Uniqueness theorems and ideal structure for Leavitt path algebras, \emph{Journal of Algebra}, \textbf{318} (2007), 270--299.

%\bibitem{tomf:lpawciacr} M. Tomforde, Leavitt path algebras with coefficients in a commutative ring, \emph{Journal of Pure and Applied Algebra}, \textbf{215} (2011), 471--484.

%\bibitem{y:thrcdo} I. Yengui, The Hermite ring conjecture in
%dimension one, \emph{Journal of Algebra},
%\textbf{230} (2008), 447--441.
\end{thebibliography}
